\documentclass[11pt,leqno]{amsart}
\usepackage{mathrsfs,dsfont,amsmath,mathrsfs,stmaryrd,enumerate}
\usepackage[all,2cell]{xy} \UseAllTwocells \SilentMatrices
\usepackage[T1]{fontenc}
\usepackage{times}
\usepackage[unicode]{hyperref}

\newtheorem{thm}{Theorem}[section]
\newtheorem{lem}[thm]{Lemma}
\newtheorem{cor}[thm]{Corollary}
\newtheorem{prp}[thm]{Proposition}

\theoremstyle{remark}
\newtheorem{rmk}[thm]{Remark}

\theoremstyle{definition}
\newtheorem{dfn}[thm]{Definition}

\newtheorem{exs}[thm]{Examples}

\numberwithin{equation}{section}

\newcommand{\aro}{\longrightarrow}
\newcommand{\arou}[1]{\stackrel{#1}{\longrightarrow}}

\newcommand{\mm}[1]{\mathrm{#1}}

\newcommand{\bb}[1]{\mathbf{#1}}

\newcommand{\cc}[1]{\mathcal{#1}}

\newcommand{\cb}{\cc{B}}

\newcommand{\cC}{\cc{C}}

\newcommand{\ce}{\cc{E}}

\newcommand{\cf}{\cc{F}}

\newcommand{\ci}{\cc{I}}

\newcommand{\ck}{\cc{K}}

\newcommand{\cl}{\cc{L}}

\newcommand{\cm}{\cc{M}}

\newcommand{\cn}{\cc{N}}

\newcommand{\co}{\cc{O}}

\newcommand{\cR}{\cc{R}}

\newcommand{\ct}{\cc{T}}

\newcommand{\cv}{\cc{V}}

\newcommand{\g}[1]{\mathfrak{#1}}

\newcommand{\PP}{\mathds{P}}
\newcommand{\NN}{\mathds{N}}
\newcommand{\ZZ}{\mathds{Z}}

\newcommand{\al}{\alpha}

\newcommand{\be}{\beta}

\newcommand{\ga}{\gamma}

\newcommand{\om}{\omega}
\newcommand{\Om}{\Omega}

\newcommand{\te}{\theta}

\newcommand{\ph}{\varphi}
\newcommand{\Ph}{\Phi}

\newcommand{\ps}{\psi}

\newcommand{\de}{\delta}

\newcommand{\la}{\lambda}

\newcommand{\si}{\sigma}

\newcommand{\ze}{\zeta}

\newcommand{\vs}{\vspace{0.3cm}}
\newcommand{\na}{\nabla}

\newcommand{\po}{\cdot}

\newcommand{\id}{\mathrm{id}}

\newcommand{\one}{\mathds 1}

\newcommand{\ti}{\times}

\newcommand{\ot}{\otimes}
\newcommand{\otu}[1]{\underset{#1}{\otimes}}

\newcommand{\wh}{\widehat}
\newcommand{\wt}{\widetilde}
\newcommand{\ov}[1]{\overline{#1}}

\newcommand{\op}{\oplus}

\newcommand{\lip}{\varprojlim}

\newcommand{\modules}[1]{#1\text{-}\mathbf{mod}}
\newcommand{\rep}[2]{\mm{Rep}_{#1}(#2)}

\newcommand{\ega}[3]{[EGA $\mathrm{#1}_{\mathrm{#2}}$, #3]}

\begin{document}

\title[Group scheme of quotient varieties]{On the fundamental group schemes of 
certain quotient varieties}

\author[I. Biswas]{Indranil Biswas$^{*}$}

\address{School of Mathematics, Tata Institute of Fundamental
Research, Homi Bhabha Road, Mumbai 400005, India}

\email{indranil@math.tifr.res.in}

\thanks{$^*$ Partially supported by a J.C. Bose Fellowship. 
}

\author[P. H. Hai]{Ph\`ung H\^o Hai$^{**}$}

\address{Institute of Mathematics, Vietnam Academy of Science and Technology, 18 Hoang Quoc Viet Street, 10307 Hanoi, 
Vietnam}

\email{phung@math.ac.vn}

\thanks{$^{**}$ Partially supported by the project Implementation of the Agreement ``LIA FORMATH VIETNAM'' of VAST, grant number QTFR01.04/18-19 and of by Projects  ICRTM01\_2019.01-2020.05 of  the International Centre for Research and Postgraduate Training in Mathematics (ICRTM), Institute of Mathematics, VAST}

\author[J. P. dos Santos]{Jo\~ao Pedro dos Santos}

\address{Institut de Math\'ematiques de Jussieu -- Paris Rive Gauche, 4 place Jussieu, 
Case 247, 75252 Paris Cedex 5, France}

\email{joao\_pedro.dos\_santos@yahoo.com}



\date{April 27, 2024}

\begin{abstract}
In \cite{armstrong}, M. Armstrong proved a beautiful result describing fundamental 
groups of quotient spaces. In this paper we prove an analogue of Armstrong's theorem 
in the setting of $F$-divided \cite{dS07} and essentially finite \cite{nori76} 
fundamental group schemes.
\end{abstract}
\maketitle

\section{Introduction}

The goal here is to establish an analogue, in the theory of fundamental group schemes,
of a beautiful topological theorem found by M. A. Armstrong, which according to R.
Geoghegan \cite{Ge} ``is the kind of basic material that ought to 
have been in standard textbooks on fundamental groups for the last fifty years'':

\begin{thm}[{\cite[p.~299, Theorem]{armstrong}}]\label{17.10.2017--1}
Let $X$ be a path connected, simply connected, locally compact metric space.
Given a group $G$ acting discontinuously on $X$, the fundamental group of
the quotient space $G\backslash X$ is isomorphic with the quotient $G/I$, where
$I< G$ is the (necessarily normal) subgroup generated by all elements having at least
one fixed point. 
\end{thm}

The setting we have in mind is the following. Let $X$ be a variety over an algebraically closed field $k$ of positive characteristic,
on which a certain \emph{finite} abstract 
group $G$ acts (for unexplained notation, see the end of this introduction).
Under a mild condition on the orbits of $G$, it is a fact that a 
reasonable quotient of $X$ by $G$ exists in the category of varieties \cite[\S~7, 
Theorem]{mav}; denote the quotient by $G\backslash X$. Now we can ask how the
different fundamental groups 
of $X$ relate to those of $G\backslash X$. In the present work, we address this 
question in  two theories of the fundamental group scheme: the $F$-divided  \cite{dS07} and the essentially finite \cite{nori76}. See Theorem 
\ref{13.09.2016--3} and Theorem \ref{23.06.2017--1} for precise statements in these
directions.

The mechanism behind these results is worth being made conspicuous as everything 
hinges on two main ideas. The first one is very elementary and commonplace in number 
theory: all ramification of an extension is concentrated on the inertia field. We 
present a clear geometric picture of this in Section \ref{genuinely_ramified}. The 
second one is more sophisticated and based on the fact that the etale fundamental 
group has, in many cases, a mysterious control over other fundamental group schemes. 
Here this is manifested through the fact that as soon as a certain morphism $Y
\,\longrightarrow\, X$ 
realizes $\pi_1(X)$ as a quotient of $\pi_1(Y)$, then the same is true 
for the $F$-divided and for the essentially finite fundamental groups. We offer a 
careful explanation in Theorem \ref{06.07.2016--1} and Theorem \ref{03.11.2017--2}.

Let us now review the remaining sections of the paper. 

The paper is essentially divided into three parts. The first part  starts with preliminaries on Tannakian duality (Sections \ref{sect:tannakian}). As a recompense for the reader of this section, we offer a slight variation of the criterion for exact sequences first presented in \cite[Appendix]{EHS08}, see Proposition \ref{tannakian_exactness}. This variation is used throughout the paper.
Then we introduce in Section \ref{genuinely_ramified} the concept of 
genuinely ramified finite morphisms, a terminology due to Balaji and Parameswaran, 
and demonstrates that this very useful notion easily connects to a number of other 
elementary ideas from the theory of coverings (in the broad sense), see Proposition \ref{10.11.2017--1}.  

The second part is devoted to $F$-divided sheaves and the $F$-divided fundamental group scheme.
Section \ref{preliminaries_F-divided} explains how to develop the $F$-divided fundamental 
group scheme \cite{dS07} beyond the realm of smooth $k$-schemes. This is 
essentially well-known as the technique behind the main point (Lemma 
\ref{08.07.2016--4}) works in more generality than stated in its first appearance 
\cite[Lemma 6]{dS07}.

Section \ref{s29.05.2018--1} exists to fill a gap in the literature and deals with an unsurprising result expressing the $F$-divided fundamental group of a quotient in the case of a {\it free} action; see Proposition \ref{16.02.2018--1}.

Section 
\ref{F-divided_quotient} offers one of our main results, an analogue of Armstrong's theorem in the setting of 
the $F$-divided fundamental group; see Theorem \ref{13.09.2016--3}. 
\\[1ex]
{\bf Theorem}. \it 
Let $G$ be a finite group acting on the normal variety $Y$ with
$$
f\,:\,Y\,\longrightarrow\, X 
$$
being the quotient. Choose $y_0\,\in\, Y(k)$ above
$x_0\,\in\, X(k)$. Then, the \textbf{ image}\footnote{The image is defined in terms of rings as in \cite[15.1]{waterhouse}. 
This part and the   corresponding ones in the next Theorem, Theorem \ref{13.09.2016--3} and Theorem \ref{23.06.2017--1} were wrongly stated in the previous versions of this paper, including the one printed by the Tohoku Math. Joural, and   A. Langer, who is  heartily thanked here, pointed this  out  to us.} of  the induced homomorphism
\[
 f_{\natural}\,:\,\Pi^{\rm FD}(Y,y_0)\,\aro\,\Pi^{\rm FD}(X,x_0)
\]
is a closed and normal subgroup scheme, and its quotient is identified with $G/I$, where $I$ is the subgroup (necessarily normal) generated by all
elements of $G$ having at least one fixed point.
\rm\\

Starting from 
Section \ref{extending} on, we concentrate on the theory of the essentially finite 
fundamental group scheme \cite{nori76}. But the structure is very much the same as in the theory of the $F$-divided group: we begin by proposing a slight generalization of 
Nori's theory to the case of non-proper varieties (Section \ref{extending}) and 
conclude, in Section \ref{quotient_EF}, by proving an analogue of Armstrong's theorem 
for this fundamental group scheme; see Theorem \ref{23.06.2017--1} (for the definition of
CPC variety see Definition \ref{05.05.19--1}).
\\[1ex]
{\bf Theorem}. \it 
Let $Y$ be a normal CPC variety. 
Let $G$ be a finite group of automorphisms of $Y$, and write 
$$
f\,:\,Y\,\longrightarrow\, X 
$$
for the quotient of $Y$ by $G$. Choose $y_0\,\in\, Y(k)$ above
$x_0\,\in\, X(k)$. Then, the \textbf{image}  of the induced homomorphism
\[
 f_{\natural}:\pi^{\rm EF}(Y,y_0)\,\aro\,\pi^{\rm EF}(X,x_0)
\]
is a closed and normal subgroup scheme, and its quotient is identified with $G/I$, where $I$ is the subgroup (necessarily normal) generated by all
elements of $G$ having at least one fixed point.\rm \\

It should be pointed out that many results in this paper are consequences of the 
recent work of Tonini and Zhang, \cite{TZ0}, \cite{TZ1}, \cite{TZ2} (we have made 
these connections explicit at the proper places). These authors study, in 
considerable generality, Tannakian categories of sheaves and, in doing so, they 
produce not only broader conclusions, but also more comprehensive frameworks. See also \cite{ABETZ} for related results.

As a final comment, we would like to call attention to the fact that Armstrong's 
theorem (Theorem \ref{17.10.2017--1}) finds a very attractive description using 
groupoids, as R. Brown pertinently points out in \cite{brown}. Since this formalism 
is closely related to the theory of Tannakian categories, we hope to re-examine our 
findings under this light in the future.

\subsection*{Notations and standard terminology}

\begin{enumerate}[(1)]
\item 
We fix once and for all an algebraically closed field $k$ of characteristic $p>0$.
\item An algebraic scheme over $k$ is a scheme of finite type over $k$ 
\ega{I}{}{6.4.1}. 
A \emph{variety} is an integral algebraic $k$-scheme.
A \emph{curve} is a one dimensional variety. 

\item To avoid repetitions, a point in a scheme $X$, unless otherwise said, means a \emph{closed} point of $X$.

\item If $X$ is a variety, $x_0$ is a closed point in it, then $\pi_1(X,x_0)$ is the etale fundamental group of \cite{SGA1}.

\item The absolute Frobenius morphism of a $k$-scheme $X$, respectively a $k$-algebra $A$,  shall be denoted by $F_X$, respectively $F_A$. 
If no confusion is likely, the subscript will be suppressed.

\item A \emph{finite-Galois} morphism of integral schemes $f\,:\,Y\,\longrightarrow\,
X$ is a finite morphism for which the associated extension of function fields
is Galois. We refer to the pertinent Galois group by $\mm{Gal}(f)$. 

\item 
A locally free coherent sheaf is called a \emph{vector-bundle}. If $\ce\subset\cf$ is an inclusion of coherent sheaves which are also vector bundles, we say that $\ce$ is a {\it sub-bundle} of $\cf$ if $\cf/\ce$ is also a vector bundle.  

\item 
An open and dense subset of a scheme $X$ is called \emph{big} if its complement has codimension at least two.

\item All group schemes are defined over $k$  and {\it affine}. 
\item A \emph{quotient morphism} between   group schemes  is a faithfully flat morphism (this terminology comes from \cite[15.1]{waterhouse}). We will employ many times the Theorem of 14.1 from \cite{waterhouse}.

\item Our conventions on representations follow \cite[Chapter 2, Part 1]{jantzen}. 

\item Following Deligne \cite[1.2]{Del90}, a tensor category $\ct$ over $k$ is a $k$-linear 
$\ot$-category which is ACU, rigid, and abelian with $k=\mm{End}(\one)$. A category is 
neutral Tannakian if there exists, in addition, a faithful, exact functor 
$\om:\ct\aro\modules k$ \cite[1.8]{Del90}. (Note that, in the presence of $\om$, the 
condition $k=\mm{End}(\one)$ is superfluous.)

\item For a tensor category $\ct$ over $k$, two elements $a=(a_0,\ldots,a_t)$ and $b=(b_0,\ldots,b_t)$ in $\NN^t$, and an object $V\in\ct$, we write $\bb T^{a,b}V$ to denote $\oplus_{i=1}^tV^{\ot a_i}\ot \check V^{\ot b_i}$. The full subcategory of all subquotients of objects of the form $\bb T^{a,b}V$ is denoted by $\langle V\rangle_\ot$. 
\end{enumerate}

\subsection*{Acknowledgements}{IB and JPdS wish to thank the ``Laboratoire International Associ\'e'' and the  ``Indo-French Program for Mathematics'' of the CNRS as well as the  International Research Staff Exchange Scheme ``MODULI'', projet  612534 from the  Marie Curie Actions of the European commission. All three  authors profit to thank the ``\'Equipe de Th\'eorie de Nombres'' of the Institut de Math\'ematiques de Jussieu--Paris Rive Gauche  for financing the participation of IB and PHH in the jury of JPdS's ``Habilitation''; this is where this collaboration began. Finally, we are grateful for the referee's contributions  which helped us to give proper explanations to a certain number of arguments. 
}

\section{Preliminaries on Tannakian duality}\label{sect:tannakian}
In this section we shall present some applications of Tannakian duality to study homomorphisms between group schemes over a field. These are analog and enhancement of the known results in \cite{dm,EHS08}. At the end, we recall the construction of the torsor associated to a tensor functor from a representation category to the category of coherent sheaves.

For simplicity, in this work a tensor category $\ct$ over $k$ is a $k$-linear $\ot$-category which is ACU, rigid, and abelian with $k=\mm{End}(\one)$. This amounts to say that $\ot$ is a $k$-linear abelian category equipped with a bi-additive $k$-linear exact functor $\otimes:\ct\times\ct\to \ct$ and a unit object $\one$, satisfying the associativity, commutativity and unity axioms (cf. \cite{dm}). Moreover, we require the existence of the dual object to each object in $\ct$. For each $U\in \ct$, its dual will be denoted by $U^\vee$ and is uniquely determined by the isomorphism
\begin{equation}\label{7.3.2019--3}{\rm Hom}_\ct(V\otimes U,W)\cong {\rm Hom}_\ct(V,U^\vee\otimes W),\end{equation}
which is functorial in $V$ and $W$. 
Choose $V=\one$ we get the isomorphism which will be frequently used later:
\begin{equation}\label{7.3.2019--2}
{\rm Hom}_\ct(U,W)\cong {\rm Hom}_\ct(\one,U^\vee\otimes W).\end{equation}
Choose $W=\one$ and $V=U^\vee$ in \eqref{7.3.2019--3}  we get the evaluation map ${\rm ev}_U:U^\vee\otimes U\to \one$, which corresponds to the identity morphism of $U^\vee$. Twisting ${\rm ev}_U$ with the symmetry $U\otimes U^\vee\to U^\vee\otimes U$ and  using \eqref{7.3.2019--3}, we get a morphism $U\to U^{\vee\vee}$. We say that $\ct$ is rigid if this morphism is an isomorphism for any $U$. In this case we can easily deduce  the canonical isomorphisms
\begin{equation}\label{7.3.2019--1}
{\rm Hom}_\ct(U,V)\cong {\rm Hom}_\ct(V^\vee,U^\vee).\end{equation}

A (neutral) fiber functor for $\ct$ is a $k$-linear faithful exact tensor functor  $\omega:\ct\to \mathbf{Vect}_k$. A tensor category $\ct$ equipped with such an $\omega$ is called a (neutral) Tannakian category. Tannakian duality tells us that  there exists an affine group scheme $G$ over $k$, defined as $\mm{Aut}^\otimes(\omega)$, such that $\omega$ factors as an equivalence $\ct\to \mm{Rep}_k(G)$ and the forgetful functor, where $\mm{Rep}_k(G)$ denotes the category of finite dimensional $k$-linear representations of $G$.  Moreover, any affine group scheme is obtained from its representation category in this way.

If $\eta:\ct\to\ct'$ is a tensor functor between Tannakian categories, compatible with the fibre functors, then it induces a homomorphism $\eta^\#$ between the Tannakian groups in the reverse direction. This yields a bijective correspondence between tensor functors between representation categories of two   group schemes and homomorphism between them (in the reverse direction).

We shall frequently use the following Tannakian criteria for homomorphisms between (affine) group schemes (cf. \cite[Proposition~2.21]{dm}. 
Given a homomorphism $\varphi:\Pi'\to \Pi$ of affine group schemes and let $\varphi^\#$ be the restriction functor $\mm{Rep}_k(\Pi)\to \mm{Rep}_k(\Pi')$. Then {\em
\begin{enumerate}
\item $\varphi$ is faithfully flat if and only if $\varphi^\#$ is fully faithful 
and for every representation $V$ of $\Pi$ and $U\,\subset\, \varphi^\#(V)$ a 
$\Pi'$--subrepresentation, then $U$ is also invariant under the action of $\Pi$.
Hereafter we shall say that $\varphi^\#$ is closed under taking subobjects. 
\item$\varphi$ is a closed immersion if and only if every representation of $\Pi'$  is isomorphic to a
subquotient of a representation of the form of $\varphi^\#(X)$, where $X$ is a representation of $\Pi$.
\end{enumerate}}

We first present an enhanced version of 
the criterion for a morphism of affine group schemes to be 
faithfully flat.  
\begin{lem}\label{12.12.2014--2}
Let $\varphi\,:\,\Pi'\,\longrightarrow\, \Pi$ be a homomorphism of affine
group schemes over $k$. Then $\varphi$ is quotient morphism (i.e. faithfully flat) if and only if 
\begin{enumerate}
\item the functor $\varphi^\#\,:\,\mathrm{Rep}(\Pi)\,\longrightarrow\,\mathrm{Rep}(\Pi')$ is fully faithful,
\end{enumerate}
and one among the following two equivalent conditions holds:
\begin{enumerate}
\item[(2)] Let $V$ be a representation of $\Pi$ and $L\,\subset\, \varphi^\#(V)$ a 
$\Pi'$--submodule of rank one. Then $L$ is also invariant under the action of $\Pi$.
\item[(2bis)] Let $V$ be a representation of $\Pi$ and $q\,:\,\varphi^\#(V)\,\longrightarrow\, L$ a quotient $\Pi'$--module of rank one. Then $L$ also has the structure of a $\Pi$-module and $q$ is equivariant. 
\end{enumerate}

In addition, if either $\Pi$  or $\Pi'$, is pro-finite, then condition (1) is already sufficient for $\ph$ to be faithfully flat. 
\end{lem}

\begin{proof}We start by noting that (2) and (2bis) are equivalent: all that is needed is to take duals. 
In view of \cite[p.~139,~Proposition~2.21]{dm}, we only need to show that (1) and (2) 
together imply that $\varphi$ is faithfully flat. Take any $V\,\in\,\mathrm{Rep}_k(\Pi)$. 
The aforementioned result of \cite{dm} guarantees that it is sufficient to show that any 
$\Pi'$--submodule $W\,\subset\, V$ is also invariant under $\Pi$. Now, if $r\,=\,\dim_kW$, then 
$\bigwedge^rW\,\subset\, \bigwedge^rV$ is, by hypothesis, invariant under $\Pi$. This means that for 
all $k$--algebras $R$, the rank one $R$--submodule $\bigwedge^rW\otimes R\,\subset\, 
\bigwedge^rV\otimes R$ is invariant under all $g\,\in\, \mathrm{Aut}_R(V\otimes R)$ belonging to 
the image of $\Pi(R)$. The standard argument in the last paragraph in Appendix 2 on page 
152 of \cite{waterhouse} proves that this is only possible if $W\otimes R$ is invariant 
under all $g\,\in\,\Pi(R)$.

Let us now deal with the case where either $\Pi$ or  $\Pi'$ is  pro-finite and show that condition (1) is enough to show faithful flatness. Let $u:\Pi\to G$ be an algebraic   quotient of $\Pi$ and consider the following commutative diagram where $u$ and $u'$ are quotient morphisms and $i$ is a closed immersion:
\[
\xymatrix{\Pi'\ar[r]^\ph\ar[d]_{u'}&\Pi\ar[d]^u\\
G'\ar[r]_{i}&G.}
\] 
Note that, if $\Pi$ is pro-finite, then $G$ can be taken to be finite, in which case $G'$ is also finite. In case $\Pi'$ is pro-finite, then $G'$ is immediately finite. 
Now, if we endow $\co(G)$ with its right regular action and use (1), we conclude that  $\dim \co(G)^{G'}=\dim_k \co(G)^{G}=1$.
But this is only possible if $i$ is an isomorphism (say, because $\co(G)$ is locally free over $\co(G)^{G'}$ of rank $\dim_k\co(G')$ \cite[\S12, Theorem 1, p104]{mav}) and the rest of the proof follows effortlessly. 
\end{proof}

The next proposition is a version of the Tannakian criterion for the exactness of a short sequence of group schemes \cite[Theorem A.1(iii)]{EHS08}. Let us first notice that for a homomorphism $\varphi:H\to G$, with the restriction functor denoted by $\varphi^\#:{\rm Rep}_k(G)\to {\rm Rep}_k(H)$, the right adjoint to $\varphi^\#$ is the induction functor ${\rm ind}_H^G$. This functor however does not map to ${\rm Rep}_k(G)$ but rather its ind-category, i.e. the category of all representations of $G$. Below we shall consider the situation when a proper adjoint to $\varphi^\#$ exists. 

\begin{prp}\label{tannakian_exactness}
Let $\ph\,:\,H\,\longrightarrow\, G$ be a morphism of affine group schemes over $k$. 
Assume that the functor $\ph^\#\,:\,{\rm Rep}_k(G)\,\longrightarrow\,{\rm Rep}_k(H)$
has a \emph{faithful} right adjoint 
$\ph_\#\,:\,{\rm Rep}_k(H)\,\longrightarrow\,{\rm Rep}_k(G)$ 
which, in addition, is such that $\ph^\#(\ph_\#(\mathds 1))$ is a trivial object of ${\rm Rep}_k(H)$. Then $\ph$ is a closed and normal immersion. 

Moreover, if $Q$ is the cokernel of $\ph$, then ${\rm Rep}_k(Q)$ is the full subcategory of ${\rm Rep}_k(G)$ consisting of those objects on which  $H$ acts trivially. 
\end{prp}

\begin{proof} 
We first notice that the co-unit $\ph^\#\ph_\#(V)\,\longrightarrow\,V$ is an epimorphism (cf. \cite[IV.3, Theorem 1,  p.90]{maclane}). Indeed, the co-unit is a natural family of maps $\varepsilon_V:\varphi^\#\varphi_\#V\to V$ which defines the canonical isomorphism
$$\Phi:{\rm Hom}_G(M,\varphi_\#V)\xrightarrow{\cong}{\rm Hom}_H(\varphi^\#M,V),\quad
\forall M\in{\rm Rep}_k(G),\ V\in{\rm Rep}_k(H).$$
We have the following commutative diagram:
$$\xymatrix{
\varphi^\#M\ar[r]^{\varphi^\#f}\ar[rd]_{\Phi(f)}&\varphi^\#\varphi_\#V\ar[d]^{\varepsilon_V}\\ & V,}$$
for any $f:M\to \varphi_\#V$ in ${\rm Rep}_k(G)$.

Choose $M=\varphi_\#W$, $f=\varphi_\#g$ for some $g:W\to V$ in ${\rm Rep}_k(H)$, we have the following diagram with the upper triangle commutative and the outer square commutative (due to the functoriality of $\varepsilon$):
$$\xymatrix{
\varphi^\#M\ar[rr]^{\varphi^\#\varphi_\#g}\ar[rrd]_{\Phi(\varphi_\#g)}
\ar@{-->}[d]_{\varepsilon_W}&&\varphi^\#\varphi_\#V\ar[d]^{\varepsilon_V}\\ 
W\ar@{-->}[rr]_g&& V.
}$$
Hence the lower triangle also commutes. As this holds for any map $g:W\to V$ and $\varphi_\#$ is faithful,  we conclude that $\varepsilon_W$ has to be surjective. Consequently, condition of \cite[~Proposition~2.21]{dm} is fulfilled, hence the map $H\to G$ is a closed immersion.

We set out to prove that (a), (b) and (c) of \cite[Theorem A.1(iii)]{EHS08}, henceforth called simply 
conditions (a), (b) and (c), are satisfied for the diagram $H\,\longrightarrow\,G\,\longrightarrow\,Q$.  Condition (a) is assured by the construction of $Q$ from its category of 
representations. Condition (c) is guaranteed by the fact that the counit $\ph^\#\ph_\#(V)\,\longrightarrow\, V$ is an epimorphism for each 
$V$. We only need to show that (b) holds. 

 Notice that the functor $\varphi^\#$ commutes with taking dual objects. Given $M\in{\rm Rep}_k(G)$, we shall define a sub-object $M_0\subset M$ such that $\varphi^\#M_0$ is the maximal trivial sub-object of $\varphi^\#M$. Using the finiteness assumption on $\varphi_\#$ we have the following ``projection'' formula:
\begin{eqnarray*}
{\rm Hom}_H(\mathds1,\varphi^\#M)&=&{\rm Hom}_H(\varphi^\#M^\vee,\mathds1)\\
&=&{\rm Hom}_G(M^\vee,\varphi_\#\mathds1)\\
&=&{\rm Hom}_G((\varphi_\#\mathds1)^\vee,M).\end{eqnarray*}
We denote by $M_0$ the trace of $(\varphi_\#\mathds1)^\vee$ in $M$ (i.e. the sum of all images of that object in $M$). Then by assumption, $\varphi^\#\varphi_\#\mathds1$ being trivial, we conclude that $\varphi^\#M_0$ is trivial in ${\rm Rep}_k(H)$. It should be the maximal sub-object in $M$ since the functor ${\rm Hom}_H(\mathds1,-)$ is left exact and the above isomorphisms are functorial. Thus condition (b) is checked.  
\end{proof}

\bigskip 

Finally, we recall the construction of a torsor of isomorphisms between two fiber functors. Assume that $(\ct,\omega)$ is a Tannakian category with Tannakian group $G=\mm{Aut}^\otimes(\omega)$. Recall that $G$ is determined by the property: for each $k$-algebra $R$, $G(R)$ is the set of $R$-linear natural isomorphisms
$$\{\theta_U:\omega(U)\otimes R\to \omega(U)\otimes R,\quad
U\in \ct\},$$
which are compatible with the tensor structures.

Let $\eta$ be another $k$-linear exact tensor functor to 
$\mathbf{Coh}(X)$ - the category of coherent sheaves on a $k$-scheme $X$. Then there exists a $G$-torsor $p:P\to X$ representing the isomorphism between $\omega$ and $\theta$ in the following sense: for each morphism $f:Y\to X$ there is a functorial bijective map between the set ${\rm Mor}_X(Y,P)$ and the set of natural isomorphisms of sheaves on $Y$:
$$\{\theta_U:\omega(U)\otimes_k\mathcal O_Y\to f^*(\omega(U)),\quad U\in \ct\}.$$
See \cite[Theorem~3.2]{dm} or \cite[1.11]{Del90}, see also \cite{nori82}. 

In the special case, when $X$ admits a $k$-point $x$ and 
$\omega=x^*\circ\theta$, this equality yields, by means of the 
description above, a point $u\in P(k)$ sitting above $x$. On the other hand, the universal identity morphism $\mm{id}:P\to P$ yields a natural isomorphism
$$\eta_U:\omega(U)\otimes_k\mathcal O_P\to p^*(\omega(U)).$$
Assume that $\ct$ is equivalent to $\mm{Rep}_k(G)$ by means of the functor $x^*\circ\omega$. Then the above isomorphism and faithful flat descent imply that the functor $\omega$, considered as functor from $\mm{Rep}_k(G)$ to $\mm{Coh}(X)$ can be given as
$$\omega(V)=(V\otimes \mathcal O_P)^G,\quad V\in\mm{Rep}_k(G).$$

\section{Preliminaries on genuinely ramified finite morphisms}\label{genuinely_ramified}

We reinterpret a theme brought to our attention by the work 
\cite[\S~6]{Balaji-Parameswaran}. Since the underlying assumptions in 
\cite{Balaji-Parameswaran} are much too restrictive, and since a fundamental 
point goes unmentioned (which is Proposition \ref{10.11.2017--1} below), we think it well to 
interpose this section.

We remind the reader that if $\ph\,:\,N\,\longrightarrow\, M$ is a Galois-finite morphism of normal 
$k$-varieties, then $\mm{Aut}_{M}(N)=\mm{Gal}(\ph)$.

\begin{dfn}[{\cite[\S~6]{Balaji-Parameswaran}}]
Let $f\,:\,Y\,\longrightarrow\, X$ be a finite surjective morphism of $k$-varieties. We say
that $f$ is genuinely ramified if $f$ is generically etale and if   the only possible factorization of $f$ as a composition of morphisms of varieties 
\[\xymatrix{Y\ar[r]& X'\ar[r]^{\text{etale}}&X}
\]
is 
\[\xymatrix{Y\ar[r]^f& X\ar[r]^{\id}&X.}
\]

\end{dfn}

Under this terminology, we can reinterpret a (probably  well-known) exercise as follows: 

\begin{lem}\label{03.11.2017--1}Let $f\,:\,Y\,\longrightarrow\, X$ be a finite-Galois morphism between normal $k$-varieties. The following statements hold: 
\begin{enumerate}
\item Assume that ${\rm Gal}(f)$ is generated by elements having at least one fixed point. Then $f$ is genuinely ramified.

\item Write $I\triangleleft\mm{Gal}(f)$ for the subgroup generated by all elements of 
$\mm{Gal}(f)$ fixing at least one point. If $\chi\,:\,Y\,\longrightarrow\,M$ is the quotient of $Y$ by 
$I$ with $u\,:\,M\,\longrightarrow\, X$ being the canonical arrow, then $\chi$ is genuinely ramified while $u$ 
is etale.

\item If $f$ is genuinely ramified, then $\mm{Gal}(f)$ is generated by the
elements having at least one fixed point. 
\end{enumerate}
\end{lem}
 
\begin{proof}(1) 
Let $Y\stackrel{\chi}{\longrightarrow} M\stackrel{u}{\longrightarrow}X$ be a factorization
of $f$ with $u$ 
finite  and etale. Let us firstly suppose that $u$ is finite-Galois. 
Then, the canonical homomorphism ${\rm Gal}(f)\,\longrightarrow\,
\mm{Gal}(u)$ is 
surjective, so that $\mm{Gal}(u)$ is generated by the elements having at least one fixed 
point. Since no $\al\in{\rm Aut}_X(M)\setminus\{{\rm id}\}$ can have a fixed point \cite[I, Corollary 5.4]{SGA1}, we conclude that  $\mm{Gal}(u)=\{e\}$. Let us now deal with the general case; for that we shall require the construction of   the ``Galois closure'' of $u$  \cite[V, \S~4(g)]{SGA1}.

Let $\Om$ be an algebraic closure of $k(Y)$ and let   $\bb y\in Y(\Om)$, $\bb m\in M(\Om)$, and $\bb x\in X(\Om)$ be the associated $\Om$-points. If $\bb m_1=\bb m$, and $\{\bb m_1,\ldots,\bb m_d\}$ is the set of $\Om$-points of $M$ above $\bb x$, basic Galois theory and the normality of $Y$ allow us to find, for each $i$, an \emph{$X$-automorphism} $g_i\,:\,Y
\,\longrightarrow\, Y$ such that $\chi(g_i(\bb y))\,=\,\bb m_i$. (In particular, $g_1\,=\,{\rm id}_Y$.) Consequently, we have constructed a morphism of $X$-schemes 
\[
\wt\chi\,:\,Y\,\aro\, \underbrace{M\ti_{X}\cdots\ti_X M}_{d} 
\]
satisfying: 
\begin{itemize} \item The composition $\mm{pr}_1\circ\wt\chi$ is none other than $\chi$, and
\item the image of $\bb y$ is $(\bb m_1,\ldots,\bb m_d)$.
\end{itemize}
Now, the \emph{connected component} of the above fiber
product containing $(\bb m_1,\ldots,\bb m_d)$, call it $\wt M$, is a
finite-Galois etale covering of $X$. (This is explained in \cite[V, \S4(g)]{SGA1}, but
the reader should do an exercise on groups acting on finite sets). Using the
factorization $Y\stackrel{\wt \chi}{\longrightarrow}\wt M\stackrel{\wt u}{\longrightarrow}
X$, we can apply our previous result to conclude that $\wt u$ is an isomorphism, so that $u$ is also an isomorphism.

(2)  Let $I\triangleleft \mm{Gal}(\chi)$ be the
(normal) subgroup generated by the automorphisms
fixing at least one point, and write $\chi:Y\longrightarrow M$ for the quotient of $Y$
by $I$ \cite[\S~7, Theorem]{mav}. It follows that 
the canonical morphism $u:M\longrightarrow X$ is finite-Galois, and
${\rm Gal}(u)\simeq\mm{Gal}(f)/I$. We contend that 
\begin{enumerate}\item[(i)]
the morphism $u$ is etale, and
\item[(ii)] $\chi$ is genuinely ramified. 
\end{enumerate}
To prove (i), it is suffices to show that $\mm{Gal}(u)$ acts freely on the points of 
$M$ \cite[\S~7, Theorem]{mav}. Since the set of points of $M$ is just the quotient of 
the set of points of $Y$ (loc. cit.), the verification is quite
straight-forward. So let $I\po 
y\,=\,m\in M$, and let $s\,\in\,\mm{Gal}(f)$ be such that $ s ( m)\,=\,m$. This means that
$s(y)\,=\,i(y)$, 
where $i\,\in\, I$. Hence, we conclude that $s\,\in\, I$, so that $ s :M\longrightarrow
M$ is just the identity map, which completes the verification of (i).

To establish (ii), we note that by 
construction, $\mm{Gal}(\chi)\,=\,I$, so that (1) can be directly employed.

(3): This is in fact a consequence of (2). 
\end{proof}

The next result connects genuinely ramified coverings with the theory of the etale fundamental group.

\begin{prp}\label{10.11.2017--1}Let $f\,:\,Y\,\longrightarrow\, X$ be a finite and surjective morphism of $k$-varieties. The morphism $f$ is genuinely ramified if and only if the morphism between etale fundamental groups $ f_{\natural}:\pi_1(Y)\,\longrightarrow\,\pi_1(X)$ is surjective. 
\end{prp} 

\begin{proof}We shall only prove the ``only if'' clause; the verification of the other one is very simple. Hence we suppose that $f$ is genuinely ramified. 
Let $\pi_1(X)\,\longrightarrow\,\g g$ be a finite quotient, and let $\g h<\g g$ be the image of $\pi_1(Y)$. We endow $E:=\g g/\g h$ with the canonical left action of $\pi_1(X)$. Let $X'\,\longrightarrow\, X$ be the etale covering associated to $E$. Since $\pi_1(Y)$ leaves one point of $E$ fixed, it follows that the etale covering $Y'\,:=\,X'\ti_XY\,\longrightarrow\, Y$ (which is associated to the $\pi_1(Y)$-set $E$) must have a connected component isomorphic to $Y$. In this way, we obtain a section $\si\,:\,Y\,\longrightarrow\, Y'$, and then a lifting 
of $f\,:\,Y\,\longrightarrow\,X$ to $X'$. But this forces $X'\,\longrightarrow\, X$ to be an isomorphism, and we conclude that $\#E=1$. Since $\g g$ is arbitrary, we conclude that $ f_{\natural}\pi_1(Y)=\pi_1(X)$.
\end{proof}

\section{Preliminaries on $F$-divided sheaves}\label{preliminaries_F-divided}

In what follows, $X$ stands 
for a noetherian $k$-scheme; we do \emph{not} assume $X$ to be of finite type. 
We wish to develop in the following lines some bases for a theory of $F$-divided sheaves on $X$ 
by employing the method of \cite{dS07}.

\begin{dfn}The category of $F$-divided sheaves is the category $\bb{Fdiv}(X)$ such that: 
\begin{itemize}\item[]\textbf{Objects} are sequences $\{\ce_n,\si_n\}_{n\in\NN}$ where $\ce_n$ is a coherent $\co_X$-module and 
\[\si_n\,:\,F_X^*\ce_{n+1}\,\stackrel{\sim}{\longrightarrow}\, \ce_n\]
is an isomorphism.
\item[]\textbf{Arrows} between $\{\ce_n,\si_n\}_{n\in\NN}$ and
$\{\ce_n',\si_n'\}_{n\in\NN}$ are families of morphisms $\al_n\,:\,\ce_n\,\longrightarrow\,\ce_n'$ such that
$\si_n'\circ (F_X^*{\al_{n+1}})\,=\,\al_n\circ\si_n$. 
\end{itemize}
\end{dfn}

The construction of $\bb{Fdiv}$ is evidently functorial, and if $f\,:\,Y
\,\longrightarrow\, X$ is an arrow of $k$-schemes, then the obvious functor $\bb{Fdiv}(X)
\,\longrightarrow\,\bb{Fdiv}(Y)$ constructed from the pull-back functor $f^*
\,:\,\bb{Coh}(X)\,\longrightarrow\,\bb{Coh}(Y)$ is denoted by $f^\#$. 
In addition, as explained in 
\cite[2.2]{dS07},   $\bb{Fdiv}(X)$ is canonically $k$-linear.

As expected, if $X$ happens to be the spectrum of a noetherian $k$-algebra $A$, we 
shall write $\bb{Fdiv}(A)$ instead of $\bb{Fdiv}(X)$, and speak about $F$-divided 
\emph{modules}.

\begin{lem}\label{08.07.2016--4}
For any $F$-divided module $\ce\,=\,\{\ce_n\}_{n\in\NN}$ over $X$, the $\co_X$-module $\ce_0$ is
locally free. 
\end{lem}

\begin{proof}
It suffices to show that any $F$-divided module $E\,=\,\{E_n,\,\si_n\}_{n=0}^\infty$ over
a noetherian local ring $A$ is free. Let $\g m$ stand for the maximal ideal of $A$. If
$M$ is an $A$-module of finite type, we write $\mm{Fitt}_i(M)$ for the $i$-th Fitting
ideal of $M$ \cite[20.2, 492ff]{eisenbud}. Then, for each $n\,\in\,\NN$, we know that
$\mm{Fitt}_i(E_n)^{[p^m]}\,=\,\mm{Fitt}_i(E_0)$, where for an ideal $\g a\,\subset\, A$,
the notation $\g a^{[p^n]}$ stands for the ideal
of $A$ generated by the image of $\g a$ under the
homomorphism $a\,\longmapsto\, a^{p^n}$ \cite[Corollary 20.5, p. 494]{eisenbud}. We assume
that $\mm{Fitt}_i(E_0)\,\not=\,(1)$. In this case, $\mm{Fitt}_i(E_n)\,\not=\,(1)$, so that
$\mm{Fitt}_i(E_n)\,\subset\, \g m$. But then $$\mm{Fitt}_i(E_0)\,\subset\, {\g m}^{[p^n]}
\,\subset\,\g m^{p^n}\, .$$ As
the algebra $A$ is separated with respect to the $\g m$-adic topology
\cite[8.10, p.~60]{matsumura}, we conclude that $\mm{Fitt}_i(E_0)\,=\,(0)$. Hence
$E_0$ is projective \cite[Proposition 20.8, p.~495]{eisenbud}, which implies that
$E_0$ is free because $A$ is local.
\end{proof}

In \cite[Proposition 1.5]{Gi}, one finds a rather meaningful statement concerning 
$F$-divided sheaves on formal schemes which, as a consequence, asserts that, up to 
isomorphism, only the direct sums of the unit object appear in the category of 
$F$-divided modules over $k\llbracket x_1,,\ldots,x_d\rrbracket$ (cf. Corollary 1.6 of \cite{Gi}). We 
believe that another explanation, in a simpler setting, might be useful (as it will 
be in the proof of Lemma \ref{28.06.2016--3} further ahead).

\begin{lem}\label{08.05.2019--1}
Let $(A,\mathfrak m,k)$ be a complete local $k$-algebra. Then, any $F$-divided module over $A$ is isomorphic to a trivial $F$-divided module
$\{A^r,{\rm canonical}\}$.
\end{lem}

\begin{proof}
Let $\{E_n,\si_n\}_{n\in\NN}$ be an $F$-divided module over $A$. Let $\beta_n\,:\,A^r 
\,\longrightarrow\, E_n$ be an isomorphism (see Lemma \ref{08.07.2016--4}) and let
$\varphi_n\,:\,A^r\,\longrightarrow\, A^r$ correspond to $\sigma_n$ under it; said differently, each diagram
\[
\xymatrix{ A\otimes_{F,A}E_{n+1}\ar[rr]^-{\sigma_n}&& E_n\\ \ar[u]^-{F^*(\beta_{n+1})}A^r\ar[rr]_{\varphi_n}&&A^r \ar[u]_-{\beta_{n}} }
\]
commutes. 
Clearly, proceeding inductively, it is possible to choose $\beta_n$ so that
$\ph_n \equiv\mm \id \mod\mathfrak m$. 
This being so, it follows that $F^n(\ph_n)\equiv {\rm id}\mod \mathfrak m^{[p^n]}$, and hence 
the difference 
\[
\begin{split}
\varphi_0\cdot F(\varphi_1)\cdots F^m(\varphi_m)-\varphi_0\cdot F(\varphi_1)\cdots F^n(\varphi_n)=\\=\varphi_0\cdot F(\varphi_1)\cdots F^m(\varphi_m)\cdot \left[{\rm id}- F^{m+1}(\ph_{m+1})\cdots F^n(\varphi_n) \right]
\end{split}
\]
is congruent to 0 modulo $\mathfrak m^{[p^{m+1}]}$, and a fortiori modulo $\mathfrak m^{p^{m+1}}$. Since ${\rm End}(A^r)$ is complete for the $\mathfrak m$-adic topology, the limit 
\[
\Phi_0:=\lim_{n\to\infty} \varphi_0\po F(\varphi_1)\cdots F^n(\varphi_n)
\]
exists in ${\rm End}(A^r)$.
It is not hard to see that in fact $\Phi_0$ belongs to ${\rm Aut}(A^r)$. 
More generally, write 
\[
\Phi_m:=\lim_{n\to\infty} \varphi_m\cdot F(\varphi_{m+1})\cdots F^{n}(\varphi_{m+n});
\]
this is an element of ${\rm Aut}(A^r)$. Then $\varphi_m\cdot F(\Phi_{m+1})=\Phi_m$, which gives an arrow in ${\bf Fdiv}(A)$:
\[
\{\beta_n\Phi_n\}:\mathds1^r\longrightarrow \{E_n,\sigma_n\}.
\]
It only takes a moments thought to see that $\{\beta_n\Phi_n\}$ is in fact an isomorphism. 
\end{proof}

Let $\{\ph_n\}\,:\,\{\ce_n,\si_n\} \,\longrightarrow\, \{\ce'_n,\si_n'\}$ be
a morphism of $\bb{Fdiv}(X)$ and write $\cC_n\,=\,\cC\!oker(\ph_n)$, $\ci_n\,=\,\ci\!m(\ph_n)$ and $\ck_n\,=\,\ck\!er(\ph_n)$. Then, the family
$\cC_n$ is also $F$-divided, that is, there exist unique isomorphisms 
\[
\tau_n\,:\,F_X^*\cC_{n+1}\,\aro\, \cC_n
\]
rendering the diagrams
\[
\xymatrix{\ar[d]^{\si_n}F_X^*\ce_{n+1}\ar[r]^{\ph_{n+1}}& F_X^*\ce_{n+1}'\ar[d]^{\si_n'}\ar[r]& F_X^*\cC_{n+1}\ar[d]^{\tau_n}\ar[r]&0 \\ \ce_{n}\ar[r]^{\ph_{n}} & \ce_{n}'\ar[r]& \cC_{n}\ar[r]&0 }
\]
commutative.
It then follows that for every $n\,\in\,\NN$, the coherent sheaf $\cC_n$ is locally free. This
and the fact that every $\ce'_n$ is locally free together imply that $\ci_n$ is locally free (and this shows that $\ck_n$ is also locally free). Consequently, the pull-back sequences 
\[0\,\aro\, F_X^*\ck_{n}\,\aro\, F_X^*\ce_n\,\aro\, F_X^*\ce_n'
\]
are \emph{exact}; because of this, each     $\si_{n}\,:\,F_X^*\ce_{n+1}
\,\longrightarrow\,\ce_n$ induces an isomorphism $F_X^*\ck_{n+1}\,\stackrel\sim\longrightarrow\,\ck_n$.
This proves: 

\begin{prp}\label{08.07.2016--3}
The category of $F$-divided modules on $X$ is abelian, and the forgetful functor 
\[
\mathbf{Fdiv}(X)\,\aro\, \mathbf{Coh}(X)\, ,\quad \{\ce_n\}\,\longmapsto\, \ce_0
\]
is faithful and exact. \qed
\end{prp}

Now, on $\bb{Fdiv}(X)$ we have a tensor product defined term-by-term and also a unit 
element. In view of Lemma \ref{08.07.2016--4}, we can define the dual of each object 
$\ce$ of $\bb{Fdiv}(X)$ in a simple fashion. This being so, $\bb{Fdiv}(X)$ is an 
abelian rigid tensor category \cite[Definition 1.15]{dm}.

\begin{thm}\label{15.11.2017--1}
For any $k$-rational point $x_0$ of $X$, the functor $\{\ce_n\}\,\longmapsto x_0^*\,
\ce_0$ defines an equivalence between $\bb{Fdiv}(X)$ and the category of representations of a certain affine group scheme $\Pi^{\rm FD}(X,x_0)$. 
\end{thm}
\begin{proof}
It follows from Proposition \ref{08.07.2016--3} and Lemma \ref{28.06.2016--1} that the functor $\{\ce_n\}\,\longmapsto x_0^*\,
\ce_0$ is exact, as it is the combination of the forgetful functor of Proposition \ref{08.07.2016--3}
and the fibre functor at $x_0$, which is exact since it is restricted on the full subcategory of vector bundles. Now, since $x_0^*(\ce_0)=0$ implies that $\ce_0=0$ (recall that $X$ is noetherian), a standard argument shows that $x_0^*$ is faithful \cite[Ch. V, Proposition 1.1]{altman-kleiman70}. The final claim of the Theorem follows from Tannakian duality \cite[Theorem 2.11]{dm}. 
\end{proof}

The affine group scheme $\Pi^{\rm FD}(X,x_0)$ in Theorem \ref{15.11.2017--1}
is called \emph{the $F$-divided fundamental group scheme} of $X$ based at $x_0$. But even in the absence of a $k$-point on $X$, we can say something by verifying the infamous condition  (cf. \cite[1.10]{Del90})
\begin{equation}\label{ic}
\mathrm{End}(\mathds1)\,=\,k\, .
\end{equation}
We start by noting that, in our particular case, the ring $\mm{End}(\one)$ is isomorphic to
\[
\co(X)^\S:= \{a\in \co(X)\,:\,\text{for each $n\in\NN$, there exists $a_n$ such that  $a=a_n^{p^n}$ }\}.
\]
Indeed, from 
the nilpotency of the nilradical of $X$  \ega{I}{}{6.1.6, p.142}, we see that for a certain $r\in\NN$,  the nilradical $\g N$ of $\co(X)$ is $\{a\in\co(X)\,:\,a^{p^r}=0\}$. (Even if $\co(X)$ fails to be noetherian.)
If $\co(X)$ is abbreviated  to $\co$, we construct the following morphism between projective systems of abelian groups 
\begin{equation}\label{16.04.2020--1}
\xymatrix{\g F:&F^0(\co) && \ar[ll]_-{\text{inclusion}} F^1(\co) && F^2(\co)
\ar[ll]_-{\text{inclusion}} & \ar[l]\cdots 
\\
\g O\ar[u]^{\phi}:&\co \ar[u]^{F^0} && \ar[ll]^-{F}\ar[u]^{F^1}  \co && \ar[ll]^-{F}\ar[u]^{F^2}\co &\ar[l] \cdots
}
\end{equation}
Using the above observation concerning the nilradical $\g N$,  we conclude  that the projective system 
\[
\xymatrix{\mm{Ker}(\phi):&\mm{Ker}(F^0)&\ar[l]_{F}\mm{Ker} (F^1)& \ar[l]_-{F} \cdots} 
\]
satisfies the Mittag-Leffler condition \cite[II.9]{hartshorne} and has a vanishing projective limit. Hence, $\lip\phi:\lip\g O\to\lip \g F$ is an isomorphism \cite[II.9.1]{hartshorne}, which concludes our argument because $\lip\g F$ is just $\cap_nF^n(\co)$.

\begin{lem}\label{28.06.2016--1}
Let $(A,\,\g m)$ be an integral, local  noetherian $k$-algebra with residue field $k$ and field of fractions $K$. If $X\,=\,\mathrm{Spec}\,A$, then $\co(X)^\S\,=\,k$. If $A$ is normal and we write $U\,=\,\mathrm{Spec}\,K$,  then 
$\co(U)^\S\,=\,\co(X)^\S\,=\,k$. 
\end{lem}

\begin{proof}The equality $\co(X)^\S\,=\,k$ follows immediately from Krull's intersection theorem \cite[8.10]{matsumura}. Now take any $a\in \co(U)^\S$. Let $v\,:\,K^*\,\longrightarrow\,\ZZ$ be a discrete 
valuation. Then $v(a)=0$. In particular, if $a_n^{p^n}=a$, we also conclude that 
$v(a_n)=0$. Combining this with the fact that $A$ is the intersection of a family of
discrete valuation rings in $K$ dominating 
$A$ \cite[Theorem 11.5]{matsumura}, we conclude that $a\,\in\, \cap_n A^{p^n}$, which is $k$.
\end{proof}Let us note in passing the following:

\begin{cor}\label{07.11.2017--1}
Let $A$ and $K$ be as before, and write $f\,:\,\mathrm{Spec}\,K
\,\longrightarrow\,\mathrm{Spec}\,A$ for the obvious morphism. Then
\[
f^\#\,:\,{\rm Hom}_{\bb{Fdiv}(A)}(\one^r,\,\one)\otimes_AK \,\aro\,{\rm Hom}_{\bb{Fdiv}(K)}(\one^r,
\,\one)
\]
is bijective. \qed
\end{cor}

The following proposition generalizes Kindler's result 
\cite[p.~6465, Lemma~2.8]{kin15}.
\begin{prp}\label{24.06.2016--2}Let $X$ be  a normal variety. Then, for each open and
dense subset $U\,\subset\, X$, and each $u_0\,\in\, U(k)$, the homomorphism 
\[
\Pi^{\rm FD}(U,u_0)\,\aro\,\Pi^{\rm FD}(X,u_0)
\]
induced by the inclusion is a quotient morphism. 
\end{prp}

\begin{proof}We first \emph{assume} that the proposition is true for big open subsets and deduce the general case from it.

Let $U$ be an open subset of $X$. 
Let $X^{\rm reg}$ (respectively, $U^{\rm reg}$) be the open subset of regular points
of $X$ (respectively, $U$). We note that $X$ is normal because $Y$ is so. So
$X^{\rm reg}$ and $U^{\rm reg}$ are big open subsets of $X$ and $U$ respectively.
Consider the commutative diagram of affine group schemes
\[
\xymatrix{\Pi^{\rm FD}(X^{\rm reg})\ar[r]^\al&\Pi^{\rm FD}(X) \\ \Pi^{\rm FD}(U^{\rm reg})\ar[r]_\be\ar[u]^\xi& \Pi^{\rm FD}(U)\ar[u]_\ze.}
\]
Using differential operators, one comfortably shows that $\xi$ is faithfully flat
\cite[p.~6465, Lemma~2.8]{kin15}. Since $X^{\rm reg}\,\subset\, X$ and
$U^{\rm reg}\, \subset\, U$ are big open subsets, the assumption implies that
$\al$ and $\be$ are faithfully flat. Consequently, $\ze$ is also faithfully flat as
a simple application of \cite[Chapter~14]{waterhouse} demonstrates. 

We now assume that $U\subset X$ is a \emph{big} open subset and set out to verify
that the conditions (1) and (2bis) appearing in Lemma \ref{12.12.2014--2} hold. As
the restriction $\mathbf{VB}(X)\,\aro\,\mathbf{VB}(U)$
is fully faithful \cite[III, \S3]{SGA2}, so is the restriction $\bb{Fdiv}(X)\,\aro\,\bb{Fdiv}(U)$; condition (1) of Lemma \ref{12.12.2014--2} is then readily verified.

To verify the condition (2bis), we need to show that for $\ce\,\in\,\bb{Fdiv}(X)$ and each
quotient morphism 
\[\ph\,:\,\ce|_U\,\aro\,\cl\ ,\]
where $\cl$ is an object of rank one in $\bb{Fdiv}(U)$, there exists some $\wt\cl\,\in\,
\bb{Fdiv}(X)$ that extends $\cl$ and furthermore there is a morphism
$\wt\ph\,:\,\ce\,\longrightarrow\,\wt\cl$ of $\bb{Fdiv}(X)$ extending $\ph$. 
Two lemmas will be proved for that purpose. 

\begin{lem}\label{28.06.2016--3}
Let $X$, $U$, $\ce$, $\cl$ and $\ph$ be as before. Then, there exists a
line bundle $\wt\cl_0$ extending $\cl_0$, and furthermore there is a quotient morphism
\[\wt\ph_0:\ce_0\,\aro\, \wt\cl_0\] extending $\ph_0$. 
\end{lem}

\begin{proof}
Let $\PP(\ce_0)\,\longrightarrow\, X$ be the projective bundle associated to $\ce_0$. The
existence of $\cl_0$ implies that there is a section 
\[\xymatrix{
 & \PP(\ce_0)\ar[d] \\U\ar[ru]^-{\si}\ar[r] &X. 
}
\]
Let $x\,\in\, X\setminus U$ and write $D$ (respectively, $D^\circ$) for the spectrum of the
complete local ring $\wh\co_x$ (respectively, field of fractions of $\wh\co_x$). It should be
emphasized that
$\wh\co_x$ is a \emph{normal domain}; see \cite[Theorem 79, p.258]{matsumura80}.
We then arrive at a commutative diagram 
\[
\xymatrix{ D^\circ\ar[r]\ar[d] &\ar[d] U\ar[r]^\si&\PP(\ce_0)\\ D\ar[r]& X.& }
\]
Since $\ce|_D$ is isomorphic to $\one_D^{\oplus r}$, it follows that $\ce|_{D^o}$ is
isomorphic to $\one_{D^\circ}^{\op r}$; hence $\cl|_{D^\circ}$ is a quotient of
$\one_{D^\circ}^{\op r}$, which implies that $\mathcal L|_{D^\circ}$ is isomorphic
to $\one_{D^o}$ (see Corollary \ref{07.11.2017--1}). Consequently, we obtain the
dotted arrow in the following commutative diagram: 
\[
\xymatrix{ D^\circ\ar[r]\ar[d] &\ar[d] U\ar[r]^\si&\PP(\ce_0)\\ \ar @{-->}@/_2.5pc/[rru]D\ar[r]& X.& }
\]
In view of Lemma \ref{08.07.2016--2} below, there is some open neighborhood $x\,\in
\,V\,\subset\, X$ and a morphism $V\,\longrightarrow\,\PP(\ce_0)$ which extends
$D\,\longrightarrow\, \PP(\ce_0)$. Therefore, it is possible to find an extension of the
morphism $$\ph\,:\,\ce_0\vert_U\,\longrightarrow\,\cl_0$$ to $V$. As $x$ is arbitrary, we
arrive at a desired conclusion. 
\end{proof}

\begin{lem}\label{08.07.2016--2}
Let $(A,\,\g m)$ be a noetherian local ring whose completion $\wh A$ is a domain (note
that in this case $A$ is also a domain). If $K$ (respectively, $\wh K$), stands
for the fractions field of $A$ (respectively, $\wh A$), and if $h\,\in\, K$ is such that its image in $\wh K$ belongs to $\wh A$, then $h\in A$.

In particular, if $R$ is any ring and $u:R\,\longrightarrow\, K$ is a morphism whose image belongs to $\wh A$, then $u(R)\subset A$.
\end{lem}

\begin{proof}
We write $h\,=\,h_1\po h_2^{-1}$, where $h_1\,\in\, A$ and $h_2\,\in\,\g m$ (otherwise $h
\,\in\, A^\times$ and there is nothing to be done). The fact that $h\,\in\, \wh A$
means simply that $\wh A h_2\,\supset\,\wh A h_1$. Since $A\,\longrightarrow\, \wh A$ is
faithfully flat \cite[p.~62, 8.14]{matsumura}, we conclude that
$\wh{A} a\cap A \,=\,A a$ for each $a\,\in\ A$ \cite[p.~49, 7.5(ii)]{matsumura},
an hence $A h_2\,\supset\, A h_1$.
\end{proof}

\emph{Continuing with the proof of Proposition \ref{24.06.2016--2}}, we employ Lemma 
\ref{28.06.2016--3} to find epimorphisms
\[
\wt\ph_n\,:\,\ce_n\,\aro\,\wt\cl_n
\]
extending $\ph_n\,:\,\ce_n|_U\,\longrightarrow\,\cl_n$. From the fact that 
\[
\bb{VB}(X)\,\aro\,\bb{VB}(U)
\]
is fully faithful [SGA2, III.3], the isomorphisms 
\[
\tau_n\,:\,F_U^*(\cl_{n+1})\,\arou\sim\,\cl_n
\]
used in the definition of $\cl\,\in\,\bb{Fdiv}(U)$ can be extended to isomorphisms
$\wt\tau_n\,:\,F_X^*\wt\cl_{n+1}\,\longrightarrow\, \wt\cl_n$. Another application of
the fully faithfulness of $\bb{VB}(X)\,\longrightarrow\,\bb{VB}(U)$ shows
that the arrows $\wt\ph_n$ give rise to an epimorphisms of $\bb{Fdiv}(X)$.
\end{proof}
 
\begin{rmk}
The results in this section follow from the work of Tonini and Zhang: Compare Theorem 
\ref{15.11.2017--1} to \cite[Theorem I]{TZ0} and
Lemma \ref{28.06.2016--1} to 
 \cite[Proposition 6.19]{TZ0}.  
 A more general result than Proposition \ref{24.06.2016--2} appears as Theorem 3.1 in \cite{TZ1}.
It is also useful to note that  \cite[Section 3]{ES16} also presents a swift development of the basis of the theory of $F$-divided sheaves on schemes of finite type over a field. 
 
\end{rmk}

\section{The $F$-divided fundamental group scheme of a quotient: the case of a free action}\label{s29.05.2018--1}

The main ideas leading to the findings of this section are folklores and many of its forms have already been published (see \cite[Proposition 1.4.4]{Katz87} or  \cite[Theorem 2.9]{EHS08} for example).

Let $f\,:\,Y\,\longrightarrow\, X$ be a finite and etale morphism of $k$-varieties. We wish to compare $\Pi^{\rm FD}(Y)$ 
and $\Pi^{\rm FD}(X)$, and the method relies on constructing a \emph{right} adjoint 
$f_\#\,:\,\bb{Fdiv}(Y)\,\longrightarrow\,\bb{Fdiv}(X)$ to the functor $f^\#\,:\,\bb{Fdiv}(X)\,\longrightarrow\,\bb{Fdiv}(Y)$. As expected, 
$f_\#$ is built up from the usual adjoint $f_*$. Let us prepare the terrain.

Given a $k$-algebra, we write $A'$ to denote the $k$-algebra whose underlying commutative ring is $A$, but whose multiplication by an element $\la\in k$ is given through the formula $\la\bullet a=\la^{p^{-1}}a$. 
With this definition, the Frobenius morphism $F_A\,:\,A'\,\longrightarrow\, A$ is $k$-linear. 

\begin{lem}\label{16.02.2018--2}Let $A$ be a $k$-algebra and $A\,\longrightarrow\, B$ a finite and etale morphism.The following claims are true. 

\begin{enumerate}\item The natural morphism of $A$-algebras 
\[
\te:B'\otu{A',F_A}A\aro B,\qquad b'\ot a\longmapsto b'^pa
\] 
is bijective. (The reader is required to write down the associated Cartesian diagram of affine schemes.)
\item Let $N'$ be a finitely presented $B'$-module. Then the natural morphism 
\[
\te_{N'} : N'\otu{A',F_A} A\aro N'\otu{B' ,F_{B}}B,\qquad n'\ot a\longmapsto n'\ot a. 
\]
is an isomorphism of $A$-modules.
\end{enumerate}
\end{lem}

\begin{proof}(1) Let $x\,:\,A\,\longrightarrow\, K$ be a morphism of $k$-algebras, where $K$ is a field, and consider
\[
\te\ot_AK:\left(B'\ot_{A',F_A}A\right)\ot_AK\aro B\ot_AK. 
\]
Under the identification
\[(B'\ot_{A',F_A}A)\ot_AK \arou\sim \left(B'\ot_{A'}K'\right)\ot_{K',F_K}K,\]
it is clear that $\te\ot_AK$ corresponds to 
\[
\left(B'\ot_{A'}K'\right)\ot_{K',F_K}K\aro B\ot_AK,\qquad (b'\ot 1)\ot r\longmapsto b'^p\ot r.
\]
Now $B'\ot_{A'}K'$ is an etale $K'$-algebra and hence the arrow above is an isomorphism of $K$-algebras due to \cite[V.6.7, Corollary on p. V.35]{bourbaki_algebre}. We conclude that $\te$ is an isomorphism by means of Nakayama's Lemma. 

(2) This is an application of part (1) and a canonical isomorphism: 
\[
\begin{split}N'\ot_{B',F_B} B & \simeq N'\ot_{B'}\left(B'\ot_{A',F_A}A\right)
\\
&\simeq N'\ot_{A',F_A}A.
\end{split}
\] 
\end{proof}

\begin{cor}Let $f\,:\,Y\,\longrightarrow\, X$ be as above. Then, for any coherent $\co_Y$-module $\cn$, the natural base-change arrow
\[
\te_\cn: F_X^*f_*(\cn)\aro f_*(F_Y^*(\cn))
\]
is an isomorphism of $\co_X$-modules. \qed
\end{cor}

Now let $\{\cn_n,\si_n\}\in\bb{Fdiv}(Y)$. Then 
\[\si_n^f:=f_*(\si_n)\circ \te_{\cn_{n+1}}:F_X^*(f_*(\cn_{n+1}))\arou\sim f_*(\cn_n)\]
defines an object $\{f_*(\cn_n),\si_n^f\}\in\bb{Fdiv}(X)$, and in this way we arrive at an adjunction 
\[
(f^\#,f_\#): \bb{Fdiv}(X)\aro\bb{Fdiv}(Y).
\]
As $\bb{Fdiv}(X)\,\longrightarrow\,\bb{Coh}(X)$ reflects isomorphisms, the construction shows that  $f_\#$ is an exact and faithful functor and the counit 
\[
f^\#f_\#(\ce_n,\si_n)\aro(\ce_n,\si_n)
\]
is always an epimorphism (the third property is in fact a consequence of the second \cite[IV.3, Theorem 1, p.90]{maclane}). 

In addition, $f^\#(f_\#(\one))$
is the \emph{trivial} object $\one \ot_k \co(G)$ of $\bb{Fdiv}(Y)$. 

\begin{prp}\label{16.02.2018--1}
Let $f\,:\,Y\,\longrightarrow\, X$ be a finite-Galois etale covering with group $G$. 
We have an exact sequence of group schemes 
\[1\,\aro\,\Pi^{\rm FD}(Y,y_0)\,\aro\,\Pi^{\rm FD}(X,x_0)\,\aro\, G\,\aro\,1\, .
\]
\end{prp}

\begin{proof}  As we observed above, $f_\#$ is 
faithful and $f^\#f_\#(\one)$ is trivial (as an object of $\bb{Fdiv}(Y)$) so that Proposition \ref{tannakian_exactness} 
can be applied to show that $f_\natural\,:\,\Pi^{\rm FD}(Y,y_0)\,\longrightarrow\,\Pi^{\rm FD}(X,x_0)$ is a closed and 
normal immersion.
   In addition,  $Q:=\mm{Coker}(f_\natural)$ is given by the category 
$$\bb{Fdiv}(Y/X):=\{\cv\in\bb{Fdiv}(X)\,:\,\text{$f^\#\cv\in\bb{Fdiv}(Y)$ is trivial}\}.$$

Let us now show that $Q\simeq G$. Given $n$, write $\cb_n=f_*(\co_Y)$ and $\ph_n:\cb_{n+1}\aro\cb_n$ for the Frobenius morphisms so that we have isomorphisms 
\[
\ph_n':\co_X\otu{F_X,\co_X}\cb_{n+1}\stackrel\sim\aro\cb_n
\] 
defining $f_\#(\one)$. Clearly, $G$ acts on each $\cb_n$ and every $\ph'_n$ is $G$-equivariant. Let now $V\in\rep kG$. We then obtain an object $\cl_{Y}(V)\in\bb{Fdiv}(X)$   by putting $\cl_{Y}(V)_n=\left(\cb_n\ot_kV\right)^G$ and defining 
\[
\co_X\otu{F_X,\co_X}\left(\cb_{n+1}\ot_kV\right)^G\arou\sim(\cb_n\ot_kV)^G
\] 
by means of $\ph_n$. (That this arrow is an isomorphism is seen by base-changing via $\co_X\to\cb_n$.) Needless to say, the $\co_X$-module $\cl_Y(V)_n$ is just the associated sheaf $Y\ti^GV$ \cite[Part I, 5.8]{jantzen}. Such a construction gives us a $k$-linear, faithful, exact and tensor functor
\[
\cl_Y:\rep kG\aro \bb{Fdiv}(X).
\]
In addition, $\cl_Y$ takes values in $\bb{Fdiv}(Y/X)$. We therefore obtain a morphism of group schemes 
\[\la:
Q\aro G. 
\]

Using that for each affine open subset $U\subset X$ the intersection $\cap_n\cb_n(U)$ is just $k$---one works on the local rings of closed points and applies Krull's intersection theorem---we conclude that $\cl_Y$ is full; the final assertion in the statement of Lemma \ref{12.12.2014--2} then assures   that $\la$ is faithfully flat. 

For each $\cv\in\bb{Fdiv}(Y/X)$, the unit of the adjunction $(f^\#,f_\#)$, \[\cv\aro f_\#f^\#(\cv),  \]
is clearly a monomorphism because $f$ is faithfully flat. Now $f_\#f^\#(\cv)\simeq f_\#(\one)^r$ and $f_\#(\one)$ is naturally an element of the essential image of $\cl_Y$ since it is the image of the left-regular representation. This demonstrates that every object of $\bb{Fdiv}(Y/X)$ is a subobject of some  $\cl_Y(V)$ and hence that 
$\la$ is a closed immersion \cite[Proposition 2.21, p.139]{dm} which concludes the verification that $Q\simeq G$.
\end{proof}

\section{The $F$-divided group scheme of a quotient}\label{F-divided_quotient}

Our aim in this section is to prove the following:

\begin{thm}\label{13.09.2016--3}
Let $G$ be a finite group acting on the normal variety $Y$ with
$$
f\,:\,Y\,\longrightarrow\, X 
$$
being the quotient \cite[Theorem, \S~7]{mav}. Choose $y_0\,\in\, Y(k)$ above
$x_0\,\in\, X(k)$. Then,  the \textbf{image}   of the  induced homomorphism
\[
 f_{\natural}\,:\,\Pi^{\rm FD}(Y,y_0)\,\aro\,\Pi^{\rm FD}(X,x_0)
\]
is a closed and normal subgroup scheme, and its quotient is
is identified with $G/I$, where $I$ is the subgroup (necessarily normal) generated by all
elements of $G$ having at least one fixed point.
\end{thm}

The proof is a simple consequence of Theorem \ref{06.07.2016--1} below, which
then becomes the main object of our efforts.

\begin{proof}[Proof of Theorem \ref{13.09.2016--3}]Let $\chi\,:\,Y\,\longrightarrow\, M$ be the quotient of $Y$ by $I$ and
$u\,:\,M\,\longrightarrow\, X$ the morphism induced by $f$. Then $u$ is finite-Galois and \emph{etale} \cite[Theorem, \S7]{mav}, ${\rm Gal}(u)\,\simeq\, G/I$, and $\chi$ is genuinely ramified (see Lemma \ref{03.11.2017--1}-(2)). Now
Theorem \ref{06.07.2016--1} guarantees that the homomorphism
\[
\chi_\natural\,:\,\Pi^{\rm FD}(Y,y_0)\,\aro\,\Pi^{\rm FD}(M,\chi(y_0)) 
\]
is faithfully flat, so that the cokernel of $ f_{\natural}$ is just the cokernel of $u_{\natural}$,
which is $G/I$ as guaranteed by Proposition \ref{16.02.2018--1}. 
\end{proof}

\begin{thm}\label{06.07.2016--1}
Let
\[
f\,:\,Y\,\aro\, X
\]
be a finite Galois, genuinely ramified morphism between normal $k$-varieties. If
$y_0\in Y(k)$ is taken by $f$ to $x_0\in X(k)$, then 
\[
 f_{\natural}\,:\,\Pi^{\rm FD}(Y,y_0)\,\aro\,\Pi^{\rm FD}(X,x_0)
\] 
is a quotient morphism. 
\end{thm}

\begin{proof}[Proof of Theorem \ref{06.07.2016--1}]
The proof of Theorem \ref{06.07.2016--1} will rely on 
   Lemma  \ref{24.06.2016--3}, which checks condition 
(1) in Lemma \ref{12.12.2014--2}, and on  Proposition 
\ref{24.06.2016--2}, which gives us the means to deduce checking condition (2) in Lemma \ref{12.12.2014--2}
from the case of etale morphisms. 
 
\begin{lem}\label{24.06.2016--3}The pull-back functor
\[
f^\#\,:\,\bb{Fdiv}(X)\,\aro\,\bb{Fdiv}(Y)
\]
is full. 
\end{lem}

\begin{proof} 
 
The proof of \cite[p.~11, Lemma~2.8]{Gi} contains a proof of this claim. Let us give 
more details. It is enough to show that for any $\ce\in\bb{Fdiv}(X)$, any arrow 
$\{\si_n\}\,:\,\one\,\longrightarrow\, f^\#\ce$ in $\bb{Fdiv}(Y)$ comes from an arrow
$\one\,\longrightarrow\, \ce$ in 
$\bb{Fdiv}(X)$. The heart of the matter is: 

\noindent{\it Claim.} Let $A$ be a normal domain which is a $k$-algebra of finite type. Let $A\to B$ be an injective morphism of finite type to a domain $B$. 
  Let $\{E_n\}\in\bb{Fdiv}(A)$ and let $\{\si_n\in B\ot_AE_n\}$ define a morphism $\one\to\{B\ot_AE_n\}$. Assume that $E_0$ is free on the basis $\{e_i\}_{i=1}^r$ and let $\si_0=\sum_ib_i\ot e_i$ with $b_i\in B$. Let $R=A[b_i]$. Then $A\to R$ is an etale morphism. 

Let $y:B\to k$ define a section to $k\to B$. We write $\g m$, $\g p$, and $\g n$ for the kernels of $y|_A$, $y|_R$ and $y$. It is a simple matter to show that 
\begin{equation}\label{09.05.2019--1}\g p=\g mR+(b_1-y(b_1),\ldots,b_r-y(b_r)).\end{equation}
Let $\wh A$, $\wh R$ and $\wh B$ stand for the $\g m$-adic completion of $A$, the $\g p$-adic completion of $R$ and the $\g n$-adic completion of $B$. 
It then follows that \[\g p\wh R=\g m\wh R+(b_1-y(b_1),\ldots,b_r-y(b_r)).\]
According to Lemma \ref{08.05.2019--1}, 
it is possible to find $\ph_1,\ldots,\ph_r\in \wh A\ot_A E_0$ such that each 
\[
\bigoplus_{i=1}^rk\po \ph_i=\bigcap_n\mm{Im}\,\wh A\ot_A E_n\aro \wh A\ot_A E_0.
\] 
Hence, the image of $\si_0$ in $\wh B\ot_A E_0$ actually belongs to the image of $\wh A\ot_AE_0$ and we conclude that the image of $b_i$ in $\wh B$ belongs to the image of $\wh A$. 
From eq. \eqref{09.05.2019--1}, we can say that $\g p\wh R$ is generated by $\g m$, so that $\g pR_{\g p}=\g mR_{\g p}$. This means that  $A\to R$ is unramified at $\g p$. As $y$ is arbitrary, it follows that $A\to R$ is etale \cite[Theorem I.3.20,p.29]{milne80}, which concludes the proof of the claim.

Now, we note that the algebra $R$ constructed in the Claim is simply the local version of the following. Regard $\si_0$ as an $\co_X$-linear morphism $\ce^\vee\ot_{\co_X}\co_Y\to\co_Y$ (we write $\co_Y$ instead of $f_*\co_Y$ to ease notation); we then obtain an $\co_X$-linear morphism $\ce^\vee\to \co_Y$ and hence an $\co_X$-linear morphism of $\co_X$-algebras $\mm{Symm}(\ce^\vee)\to\co_Y$; write $\mathcal R$ for the image of this arrow. If $U=\mm{Spec}\,A$ is an affine and open subset of $X$ and $B=\co_Y(U)$, we conclude that, in the terminology of the Claim, $R$ is simply $\cR(U)$. Then, as $Y\to X$ is genuinely ramified, we can say that $\cR=\co_X$, which means that $\si_0\in H^0(X,\ce_0)$. Applying this fact with each $\si_n$, we obtain the desired result. 
\end{proof}

\emph{We put $G\,:=\,\mm{Gal}(f)$ and set out to verify that condition (2) 
of Lemma \ref{12.12.2014--2} is valid.} Take $\ce\,\in\,\bb{Fdiv}(X)$, and let $\cl\,\longrightarrow\,f^\#\ce$ be a 
sub-object of rank one. Let now $\ci\,\subset\, f^\#\ce$ be the image of $\bigoplus_gg^\#\cl$ in $f^\#\ce$ (this is an 
object of $\bb{Fdiv}(Y)$). Then $\mathcal I$ is invariant under the action of $G$. Hence if we choose $U$ and $V$ as in the proof of Lemma \ref{24.06.2016--3} then $\mathcal I_V$ (the restriction of $\mathcal I$ to $V$) descends to an $F$-divided sheaf on $U$, call it $\mathcal M$. Then there is  an isomorphism 
\[ (f|_V)^\# \cm \,\arou\sim\, \ci_V\, .\] 
 We now consider the commutative 
diagram of affine group schemes obtained by choosing $y_0\,\in\, V(k)$ and $x_0\,:=\,f(y_0)$: \[ \xymatrix{\Pi^{\rm 
FD}(V,y_0)\ar[d]\ar[rr]&&\Pi^{\rm FD}(Y,y_0)\ar[d] \\ \Pi^{\rm FD}(U,x_0)\ar[rr]&&\Pi^{\rm FD}(X,x_0).} \] Since the 
horizontal arrows are all faithfully flat (Proposition \ref{24.06.2016--2}), the subspace $y_0^*(\ci)\,=\,x_0^*(\cm)$ 
remains fixed under the action of $\Pi^{\rm FD}(X,x_0)$ on $x_0^*(\ce)$ because it remains fixed under the action of 
$\Pi^{\rm FD}(U,x_0)$. It then follows that there exits $\ov\cm\,\in\,\bb{Fdiv}(X)$ and an isomorphism 
\[f^\#\ov\cm\,\arou\sim\,\ci\, .\] We now observe that $y_0^*(\ci)$ is a semisimple $\Pi^{\rm FD}(Y,y_0)$-module, and 
hence there is an epimorphism of $\Pi^{\rm FD}(Y,y_0)$-modules:\[\ps\,:\,\ci\,\aro \,\cl\, . \] Consequently, 
$\cl\,=\,\ci\!m(\ps)$; since $\ps$ is an arrow from $\ci\,=\, f^\# \ov\cm$ to $f^\# \ce$, Lemma  \ref{24.06.2016--3} 
tells us that $\ps$ is the image of an arrow $\ov\cm\,\longrightarrow\,\ce$, so that $\cl$ belongs to the essential image 
of $f^\#$. This completes the proof of Theorem \ref{06.07.2016--1}. \end{proof}

\begin{rmk}\label{23.06.2017--2}
It is essential that $X$ be a \emph{normal} variety for Proposition 
\ref{24.06.2016--2} to hold. Indeed, pick $x_1,x_2$ distinct closed points of 
$\PP^2_k$ and let $\ps\,:\,\PP^2_k\,\longrightarrow\, X$ be the identification of 
them \cite[Theorem 5.4, p.570]{Ferrand}. We know that on $\PP^2\setminus\{x_1,x_2\}$, which is thought of as an open 
and dense subset of $X$, there are no non-trivial $F$-divided sheaves (see Proposition 2.7 and then Theorem 2.2 in \cite{Gi}). On the other 
hand, we now show that $\Pi^{\rm FD}(X)\not=0$ by exhibiting a non-trivial 
$F$-divided sheaf of rank one. 

Take the trivial line bundle ${\mathbb L}\,:=\, {\mathcal O}_{\PP^2_k}$ and any
$\lambda\, \in\, k^*\,=\, k\setminus\{0\}$. Identify the fiber ${\mathbb L}_{x_1}\,=\,k$
with ${\mathbb L}_{x_2}\,=\,k$ by sending any $c\, \in\, {\mathbb L}_{x_1}$ to
$c\lambda\, \in\, {\mathbb L}_{x_2}$. This produces a line bundle on
$X$, which we will be denoted by ${\mathbb L}(\lambda)$. Then we have
$F^*{\mathbb L}(\lambda)\,=\, {\mathbb L}(\lambda^p)$, where $p$ is the characteristic
of $k$. Since $k$ is perfect, the line bundle ${\mathbb L}(\lambda)$ admits an
$F$-divided structure. On the other hand, ${\mathbb L}(\lambda)$ is nontrivial
if $\lambda\,\not=\,1$.
\end{rmk}

\section{Extending the Tannakian interpretation of the essentially finite fundamental group}\label{extending}

In this section we set out to extend Nori's theory of essentially finite 
vector-bundles \cite{nori76} to a slightly larger class of $k$-schemes other 
than the proper ones. It should be noted that such an enterprise has received 
attention from other geometers in recent times (cf. \cite{BV} and \cite{TZ2}) and our 
modest contribution is to throw light on a variation of Nori's initial method: connect 
points via projective curves. It is a simple alternative to the more formal condition 
on global sections of vector bundles \cite[Definition 7.1]{BV}.

\begin{dfn}\label{05.05.19--1}Let $X$ be an algebraic  $k$-scheme. A chain of proper curves on $X$ is a 
family of morphisms from proper curves $\{\ga_i\,:\,C_i\,\longrightarrow\, X\}_{i=1}^m$ such 
that the associated closed subset $\cup_i\,\mathrm{Im}(\ga_i)$ is connected. For the sake of 
brevity, we shall refer to the closed subset $\cup_i\,\mathrm{Im}(\ga_i)$ as ``a chain of 
proper curves''.

The algebraic $k$-scheme is said to be connected by proper chains (CPC for short) if any two points belong to a chain of proper curves. 
\end{dfn} 

As Ramanujam's Lemma guarantees (see the Lemma of Section 6, p.~56 of \cite{mav}), each proper and connected $k$-scheme is connected by proper chains (CPC). Another easily available way to produce CPC schemes is to note: 
\begin{lem}Let $f:Y\longrightarrow X$ be a surjective morphism of algebraic $k$-schemes. Then, if $Y$ is CPC, is $X$ likewise. \qed
\end{lem}

Also, with a bit more work, we can say: 

\begin{lem} Let $U\,\longrightarrow \,X$ be an open embedding of $k$-schemes. Assume that $X$ is projective and that $U$ is big in $X$. Then, if $U$ is connected, it is CPC. 
\end{lem}

\begin{proof}We give a proof by assuming $U$ irreducible; the general case can easily be obtained from this. As $U$ is assumed to be dense in $X$, this implies that $X$ is also irreducible. Let $d$ stand for the dimension of $X$. If $d=1$, then $U_{\rm red}$ is a proper curve, so that nothing needs to be demonstrated. We now proceed by induction on $d$, which is then $\ge2$.

Let us denote by $Z$ the reduced closed subscheme of $X$ whose underlying topological 
space is $X\setminus U$. Given distinct points $x,y$ of $U$, let $\ph\,:\,X'\,
\longrightarrow\, X$ stand for the blowup of $\{x,y\}$; the inverse image of $U$
(respectively, $Z$) is denoted by $U'$ (respectively, $Z'$). Clearly, $\ph:Z'\longrightarrow Z$ is an isomorphism and $\mathrm{codim}(Z',X')\ge2$. This being so, $U'$ is big in $X'$. Let us now choose a closed immersion
$X'\,\longrightarrow\,\PP^N$ and then a hyperplane $H\subset\PP^N$ enjoying the following properties
\begin{enumerate}\item[a)] both and $H\cap U'$ and $H\cap X'$ are irreducible of dimension $d-1$. (Apply
(1) and (3) of \cite[Corollary 6.11]{jouanolou}.)
\item[b)] The intersections $H\cap \ph^{-1}(x)$ and $H\cap \ph^{-1}(y)$ are non-empty. (Apply (1b) of loc. cit.)
\item[c)] If $W'\,\subset\, Z'$ has dimension zero, then $H\cap W'=\emptyset$, and if $W'\subset Z'$ is an irreducible component of dimension $\ge1$, we have $\dim H\cap W'=\dim Z'-1$. (Apply (1a) and then (3) and (1b) of loc. cit.)
\end{enumerate}
We then see that $H\cap U'\,\subset\, H\cap X'$ is big and of dimension $d-1$. Hence, if 
 $x'\,\in\, H\cap\ph^{-1}(x)$ and $y'\,\in\, H\cap\ph^{-1}(y)$, there exists a chain of proper curves containing them both and a fortiori there exits a chain of proper curves containing $x=\ph(x')$ and $y=\ph(y')$.
\end{proof}

\begin{rmk}We do not know if a CPC variety is, unconditionally, a big open subset of a proper variety. 
\end{rmk}

\begin{dfn}Let $\ce$ be a vector-bundle on an algebraic $k$-scheme $X$. We say that $\ce$ is Nori-semistable if, for each morphism from a smooth projective curve $\ga\,:\,C\,\longrightarrow\, X$, the vector-bundle $\ga^*\ce$ is semistable \cite[p. 14]{seshadri} and of degree zero. The category of Nori-semistable vector-bundles on $X$ is the full subcategory of $\bb{VB}(X)$ having as objects the Nori-semistable ones. This category will be denoted by $\bb{NS}(X)$.
\end{dfn}

\begin{rmk}It is not (yet) universally accepted to call the above defined vector-bundles  {Nori-semistable}. This is because in \cite{nori76}, the author introduces a slightly different category, viz. he restricts attention to those morphisms $\ga\,:\,C\,\aro\, X$ which are birational to the image in $X$. 
\end{rmk}

The next result is a straightforward application of Nori's method in \cite[\S~3]{nori76}.

\begin{lem}\label{16.01.2019--1}
Let $X$ be a reduced algebraic $k$-scheme which is connected by proper chains. Let $\ph\,:\,\ce\,
\longrightarrow\,\cf$ be a morphism of $\bb{NS}(X)$. Then:
\begin{enumerate}\item The function 
\[
r_\ph\,:\,x\in X(k)\longmapsto \mm{rank}\, x^*\ph
\]
is constant. 
\item Both $\ck \!er(\ph)$ and $\ci\!m (\ph)$
are vector-bundles. Furthermore, they are Nori-semistable. 
\end{enumerate}
\end{lem}

\begin{proof}(1) Let $\ga\,:\,C\,\longrightarrow\, X$ be a morphism from a \emph{smooth} proper
curve and let $t_1$ and $t_2$ be $k$--points of $C$. Since $\ga^*\ce$ and $\ga^*\cf$ are semistable vector-bundles
on $C$, we know that 
\[
\mm{rank}\,t_1^*[\ga^*(\ph)]\,=\,\mm{rank}\,t_2^*[\ga^*(\ph)]
\]
(this is a simple exercise \cite[Proposition 8, p.~18]{seshadri}).
But         $t_i^*[\ga^*(\ph)]\,=\,[\ga(t_i)]^*\ph$,  so that $r_\ph$ is constant on
$\mathrm{Im}(\ga)$. Now, if instead of assuming $C$ to be smooth we simply take it to be projective, the same conclusion can be achieved by considering the normalization. 
Since by definition any two points of $X$ can be joined by a chain of projective curves, $r_\ph$ is allover constant.

(2) Granted (1), \cite[Exercise 5.8, Chapter II]{hartshorne} shows that   $\cC\!oker(\ph)$ is a vector-bundle. Employing the exact sequence 
\[
0\,\aro\,\ci\!m(\ph)\,\aro\,\cf\,\aro\, \cC\!oker(\ph)\,\aro\,0\, ,
\]
we see that $\ci\!m(\ph)$ is a vector-bundle. An analogous reasoning shows that $\ck\!er(\ph)$ is a vector-bundle. 

In order the prove that $\ci\!m(\ph)$ and $\ck\!er(\ph)$ are in $\bb{NS}(X)$, we give ourselves a morphism $\ga:C\aro X$ from a smooth and proper curve. Since $\ga^*\ci\!m(\ph)=\ci\!m(\ga^*(\ph))$ and  $\ga^*\ck\!er(\ph)=\ck\!er(\ga^*(\ph))$, a direct application of \cite[Proposition 8, p.18]{seshadri} suffices.  
\end{proof}

\begin{thm}\label{16.01.2019--3}Let $X$ be a reduced algebraic $k$-scheme which is connected by proper chains.
 Then, once a point $x_0\in X(k)$ is chosen, the category $\bb{NS}(X)$, together with the functor $x_0^*\,:\,\bb{NS}(X)
\,\longrightarrow\, \modules k$, is neutral Tannakian. 

In particular,   for each $\ce\in\bb{NS}(X)$, the vector space $H^0(X,\ce)$ is finite dimensional. 
\end{thm}

\begin{proof}Lemma \ref{16.01.2019--1}(2) shows that $\bb{NS}(X)$ is abelian and that $x_0^*$ is an exact functor. If $\ce\in\bb{NS}(X)$, then   the dimension of the vector space $x_0^*\ce$ is simply the rank of $\ce$, and hence $x_0^*$ must be faithful since only the zero object can have a zero image.  It is elementary to see that Nori-semistable vector bundles must have Nori-semistable duals. 

To deal with the tensor structure, we note that for any $\ce\in\bb{NS}(X)$ and any morphism $\ga:C\aro X$ from a smooth and projective curve, the vector bundle $\ga^*\ce$ is semi-stable as well as all its pull-backs under the Frobenius morphism $F^{*i}(\ga^*\ce)$. Hence, if $\cf\in\bb{NS}(X)$, a fundamental fact of the theory says that $\ga^*(\ce)\ot\ga^*(\cf)$ is semi-stable \cite[Theorem 3.23]{ramanan-ramanathan84}. 

Finally, the last claim follows from the fact that $H^0(X,\ce)=\mathrm{Hom}_{\co_X}(\co_X,\ce)$ is a subspace of ${\rm Hom}_k(k,x_0^*\ce)$.
\end{proof}

\begin{rmk} In \cite[Definition 7.1]{BV}, the authors introduce the notion of a pseudo-proper $k$-scheme (or stack) by requiring that the space of global sections of any vector bundle is of finite dimension over $k$. In view of the final claim in Theorem \ref{16.01.2019--3}, the CPC condition is, {\it a priori},   weaker. 
\end{rmk}

\begin{dfn}The group scheme associated to $\bb{NS}(X)$ via the fibre functor $x_0^*$ is denoted by $\pi^{\rm 
S}(X,x_0)$. \end{dfn}

The category $\bb{NS}(X)$ is rather large and its understanding is less sound than that of its ``largest pro-finite quotient'', now our main topic of interest. In order to present the theory in a different light, we shall make a brief digression which is certainly well-known to the cognoscenti (see the discussion on page 331 of \cite{BV}).

Let $\ct$ be a small tensor category over $k$ (recall that we follow the convention of \cite[1.8]{Del90} concerning the usage of ``tensor category''); its set of isomorphisms classes carries an evident structure of commutative semi-ring, with addition and multiplication constructed from direct sums and tensor products. Let $\mm K(\ct)$ stand for the associated commutative ring \cite[I.2.4]{bourbaki_algebre}. (Warning: this is \emph{not} the Grothendieck ring obtained by killing extensions.)  We say, following \cite[Definition on p.35]{nori76}, that $V\in \ct$ is \emph{finite} if its class in $\mm K(\ct)$ is integral over the prime sub-ring. 
Said differently, $V\in\ct$ is finite if 
there exits a set of non-negative integers $\{a_i,b_i\,:\,i=0,\ldots,s\}$  such that 
\begin{equation}\label{11.12.2018--1}
V^{\ot s+1}\op\bigoplus_{i=0}^s(V^{\ot  i})^{\op a_i}\simeq\bigoplus_{i=0}^s (V^{\ot  i})^{\op b_i}.
\end{equation}
By standard knowledge from the theory of integral extensions \cite[Theorem 9.1]{matsumura}, we see that the set of finite objects is stable by tensor products and direct sums. Without much effort, one sees also that the dual of a finite object is necessarily finite. We then define (following \cite[Definition on pp. 37-8]{nori76}) the \emph{essentially finite category}, $\bb{EF}(\ct)$, as the \emph{full} subcategory of $\ct$ whose objects are 
\[
\left\{V\in\ct\,:\,\begin{array}{c}\text{there exists a finite object $\Ph$ and sub-objects}\\ \text{$V''\,\subset\, V'\,\subset
\,\Ph$ such that $V\,=\,V'/V''$}\end{array}\right\}.
\]

The following results can be found in \cite{nori76}, and are gathered as a Lemma to help our arguments. 

\begin{lem}\label{17.01.2019--1}Suppose that $V$ is finite and satisfies  \eqref{11.12.2018--1}. Write \[T=V^{\ot0}\op\cdots\op V^{\ot s}. \]
\begin{enumerate}[(1)]
\item For each $\ell\in\NN$, we can find  $m$ and a   monomorphism $V^{\ot \ell}\aro T^{\op m}$.
\item For each $\la\in\NN$, we can find $\mu$ and a   monomorphism $\check V^{\ot \la}\aro\check T^{\op\mu}$. 
\item  For each couple $\ell,\la$ of non-negative integers, we can find $n\in\NN$ and a   monomorphism $V^{\ot m}\ot \check V^{\ot \la}\aro (T\ot \check T)^{\op n}$. 
\item Let $a$ and $b$ be belong to $\NN^t$,  and let $W$ be a   subobject of   $\bb T^{a,b}V$. Then, $W$ is a   subobject of some $(T\ot\check T)^{\op n}$. 
\end{enumerate}
\end{lem}

\begin{proof}(1) We proceed by induction on $\ell$, and note that the case $\ell=1$ is trivial. Suppose that $V^{\ot \ell}\aro T^{\op m }$ is a monomorphism so that $V^{\ot \ell+1}\to T^{\op m }\ot V$ is also a monomorphism \cite[Proposition 1.16, p.119]{dm}. Now, if $b\ge mb_i$ for all $i$,  
\[\begin{split}T^{\op m }\ot V & = \left(V^{\ot1}\op\cdots\op V^{\ot s}\right)^{\op m }\op( V^{\ot s+1})^{\op m }
\\
&\aro T^{\op m}  \op \left(\bigoplus_{i=0}^s (V^{\ot i})^{\op b_i}\right)^{\op m}
\\&\aro  T^{\op m}\op T^{\op b},
\end{split}
\]  
where the above arrows are monomorphisms. 

(2) The same proof as before works. 

(3) Employing the notation from previous items, we   have a   monomorphism $V^{\ot \ell}\ot\check V^{\ot \la}\aro T^{\op m}\ot \check T^{\op \mu}=(T\ot\check T)^{m\mu}$.

(4) Immediate from (3).  
\end{proof}

\begin{prp}\label{13.11.2017--1}Let us maintain the above notations and conventions. Also, we give ourselves a group scheme $G$ over $k$ (affine, as always) and let $\bb{EF}_G$ stand for $\bb{EF}(\mm{Rep}_k(G))$. 
\begin{enumerate}[(i)]
\item The category $\bb{EF}(\ct)$ is stable under subobjects, quotients,  tensor products and duals. 
\item  If $G$ is finite, then  $\bb{EF}_G=\mm{Rep}_k(G)$. 
\item If $G$ is pro-finite, then $\bb{EF}_G=\mm{Rep}_k(G)$.
\item Suppose that $\bb{EF}_G=\rep kG$. Then $G$ is pro-finite.  
\item Let $\om:\mm{Rep}_k(G)
\longrightarrow\modules k$ be the forgetful functor and $\om^{\rm EF}$ its restriction to $\bb{EF}_G$.
Then $\bb{EF}_G$ and $\om^{\rm EF}$ define a neutral Tannakian category whose associated  group scheme, call it $G^{\rm EF}$, is pro-finite. In addition, if   $\chi:G\to G^{\rm EF}$ stands for the morphism of group schemes obtained from the functor $\bb{EF}_G\to\mm{Rep}_k(G)$, then, $\chi$ 
is universal \cite[III.1]{maclane} from $G$ to the category of pro-finite affine group schemes.
\end{enumerate} 
\end{prp}
\begin{proof}
$(i)$. The verification is simple and omitted. 

$(ii)$.
Let $R$ be the left-regular representation of $G$.  Following the argument of Lemma 2.3 in \cite{nori76}, we see that  $R\ot R\simeq R^{\op r}$. Consequently,  $R$ is finite.   
Since any $V\in\rep kG$ is a subobject of a certain direct sum $R\op\cdots\op R$ (this is proved for the right-regular representation in \cite[3.5, Lemma]{waterhouse}), we are done.

$(iii)$. This can be deduced from the previous item without difficulty. 

$(iv)$. We now show that if $f:G \to H$ is a quotient morphism to an {\it algebraic} group scheme $H$, then $H$ is finite. 
Let then $W$ be a faithful representation of $H$ so that $\mm{Rep}_k(H)=\langle W\rangle_\ot$ \cite[Proposition 2.20]{dm}.  By hypothesis, there exist a finite  $V\in\mm{Rep}_k(G)$  such that   $f^\#W\in\langle V\rangle_\ot$; note that  in particular $f^\#$ takes values in $\langle V\rangle_\ot$. Now, if $T\in\langle V\rangle_\ot$ is constructed from $V$ as in Lemma \ref{17.01.2019--1}, we conclude that each element of $\langle V\rangle_\ot$ is a subquotient of some power of $T\ot \check T$. Applying Proposition 2.20 of \cite{dm}, this shows that the  group scheme associated to $\langle V\rangle_\ot$ is finite. But the latter is the source of a faithfully flat morphism to $H$ \cite[Proposition 2.21(a)]{dm}, and hence $H$ is finite. 

$(v)$. This is also left to the reader. 
\end{proof}

Given Proposition \ref{13.11.2017--1}, we can put forward the
\begin{dfn}\label{15.06.2018--2}The category of essentially finite vector-bundles is, in the above notation, $\bb{EF}(\bb{NS}(X))$. We shall abuse terminology and write $\bb{EF}(X)$ instead. If $x_0\,\in\, X(k)$, then $\pi^{\rm EF}(X,x_0)$ stands for the pro-finite group scheme constructed from Proposition \ref{13.11.2017--1}. 
\end{dfn}

\begin{rmk}\label{05.05.19--3}
We notice that $\pi^{\rm EF}(X,x_0)$ classifies finite torsors on $X$. Indeed, let $P\to X$ 
be a $G$-torsor on $X$ under a finite $k$-group scheme $G$. 
The sheaf $\mathcal O_P$ as a vector bundle on $X$ satisfies the equation 
$\mathcal O_P\otimes\mathcal O_P\cong \mathcal O_P^{\oplus n}$ where $n=\dim_k\mathcal O_G$. 
Therefore, according to \cite[Proposition~(3.4)]{nori76}, $\mathcal O_P$ is semi-stable of degree 
0 when restricted to any proper curve on $X$. We conclude that $\mathcal O_P$ is in 
 $\bb{EF}(\bb{NS}(X))$. Therefore, the functor 
 $${\rm Rep}_k(G)\to \bb{Coh}(X),\quad V\mapsto P\ti^GV,$$
  lands in $\bb{EF}(\bb{NS}(X))$.
 Tannakian duality yields a group homomorphism $\pi^{\rm EF}(X,x_0)\to G$.
\end{rmk}

Let us end this section with comments on simple structural results which shall prove useful further ahead. 
\begin{rmk} \label{05.05.19--2}
Let $X$ be a CPC variety and $\cv$ an essentially finite vector bundle on it. We then derive the existence of a finite group scheme $G$, a representation $V\in{\rm Rep}_k(G)$ and a $G$-torsor $P\to X$ such that $P\ti^GV=\cv$. Now, if $V^{(1)}$ is the Frobenius twist of  $V$ \cite[Part I, 2.16]{jantzen}, we know that  $F^*(P\ti^GV)\simeq P\ti^G V^{(1)}$. Hence, for a certain $h\in\NN$, the vector bundle  $F^{*h}(\cv)$ is of the form $P^{\rm et}\ti^{G^{\rm et}}V^{(h)}$ for a certain etale covering $P^{\rm et}\aro X$. The least integer $h$ enjoying this property is to be called the \emph{height} of $\cv$. 
\end{rmk}

\section{The essentially finite fundamental group scheme of a quotient: the case of a free action}

This section functions as did Section \ref{s29.05.2018--1} and we show---by imitating the proof of  \cite[Lemma 2.8]{EHS08}---that the exact analogue of Proposition 
\ref{16.02.2018--1} holds.

\begin{prp}\label{29.05.2018--1}
Let $f\,:\,Y\,\longrightarrow\, X$ be a finite-Galois etale covering of CPC varieties over $k$. Let $G:={\rm Gal}(f)$. 
We then have exact sequences of   group schemes 
\begin{equation}\label{18.01.2019--3}1\,\aro\,\pi^{\rm S}(Y,y_0)\,\aro\,\pi^{\rm S}(X,x_0)\,\aro\, G\,\aro\,1\, 
\end{equation}
and 
\begin{equation}\label{18.01.2019--2}1\,\aro\,\pi^{\rm EF}(Y,y_0)\,\aro\,\pi^{\rm EF}(X,x_0)\,\aro\, G\,\aro\,1\, .
\end{equation}
\end{prp}

\begin{proof} 
We start by constructing a right adjoint to $f^*:\bb{NS}(X)\aro\bb{NS}(Y)$.
 Let $\cf\in\bb{NS}(Y)$ and put $\ce:=f_*(\cf)$; this is a vector bundle on $X$. If $\al:G\ti Y\aro Y$ stands for the action morphism,  the Cartesian diagram 
\[
\xymatrix{
G\ti Y \ar[r]^-{\mm{pr}}\ar[d]_\al & Y\ar[d]^f 
\\
Y\ar[r]_f&X} 
\]
and the isomorphism $f^*f_*(\cf)\simeq \al_*(\mm{pr}^*\cf)$  show that 
\begin{equation}
\label{18.01.2019--1}
f^*(\ce)\simeq \bigoplus_{g\in G}g^*(\cf)
\end{equation}
In particular, $f^*(\ce)$ is Nori-semistable. Let us now prove that $\ce$ is Nori-semistable. 
Consider now the Cartesian diagram 
\[
\xymatrix{Y_C\ar@{}[dr]|\square\ar[r]^{\de}\ar[d]_{f_C}&Y\ar[d]^f\\C\ar[r]_\ga&X,}
\] 
where $C$ is a projective and smooth curve.
Note that $f_C:Y_C\aro C$ is a finite etale covering and hence, if $j:D\aro Y_C$ stands for the immersion of a connected component, we conclude that $f_C j$ is a finite and etale covering.  From the fact that $f^*(\ce)$ is Nori-semistable, we conclude that \[\begin{split}(f_C j)^*(\ga^*\ce)&\simeq j^*f_C^*\ga^*(\ce)\\& \simeq j^*\de^*(f^*\ce)\end{split}\] is semistable and of degree zero. So it must be the case that $\ga^*(\ce)$ is semistable and of degree zero---because the degree of a pull-back is simply multiplied by the degree of the covering---, which in turn shows that $\ce$ is Nori-semistable. We have therefore our adjoint $f_*:\bb{NS}(Y)\aro\bb{NS}(X)$, which is, in addition,   faithful and exact. Clearly, $f^*f_*(\co_Y)$ is trivial (isomorphic to $\co_Y\ot_k\co(G)$) and the twisting functor $V\,\longrightarrow\,Y\ti^GV$ defines an equivalence between ${\rm Rep}_k(G)$ and $\{\cv\in\bb{NS}(X)\,:\,\text{$f^*\cv$ is trivial}\}$; an inverse is $\cv\,\longmapsto\, H^0(Y,f^*(\cv))$ (recall that $H^0(Y,\co_Y)=k$). We then conclude, applying Proposition \ref{tannakian_exactness}, that $\pi^{\rm S}(Y,y_0)\,\longrightarrow\,\pi^{\rm S}(X,x_0)$ is a normal closed immersion and that the cokernel is isomorphic to $G$, which is exactness of \eqref{18.01.2019--3}. 

To prove   exactness of \eqref{18.01.2019--2}, we only require minor adjustments: let   $\cf$ as before be required  not only to be Nori-semistable but essentially finite. Then, isomorphism  \eqref{18.01.2019--1} tells us that $f^*f_*(\cf)=f^*(\ce)$ is essentially finite. This means that   the canonic morphism  $\pi^{\rm S}(Y,y_0)\,\longrightarrow\,\bb{GL}(x_0^*\ce)$ 
has a finite image; by exactness of \eqref{18.01.2019--3}, the arrow $\pi^{\rm S}(X,x_0)\,\longrightarrow\, \bb{GL}(x_0^*\ce)$ has also finite image. This  in turn translates into the fact that $\ce\in\bb{EF}(X)$.  Hence   $f_*:\bb{EF}(Y)\aro\bb{NS}(X)$ takes values in $\bb{EF}(X)$. 
As before, an application of Proposition \ref{tannakian_exactness} establishes   exactness of \eqref{18.01.2019--2}.  
\end{proof} 

\begin{rmk}We would like to draw the reader's attention to the fact that Theorem IV in \cite{ABETZ}  overlaps with Proposition \ref{29.05.2018--1}. 
\end{rmk}

\section{The essentially finite fundamental group of a quotient}\label{quotient_EF}

In this section we wish to prove the following theorem. 

\begin{thm}\label{23.06.2017--1}Let $Y$ be a normal CPC variety. 
Let $G$ be a finite group of automorphisms of $Y$, and write 
$$
f\,:\,Y\,\longrightarrow\, X 
$$
for the quotient of $Y$ by $G$ \cite[\S~7, Theorem]{mav}. Choose $y_0\,\in\, Y(k)$ above
$x_0\,\in\, X(k)$. Then, the \textbf{image}  of the  induced homomorphism
\[
 f_{\natural}:\pi^{\rm EF}(Y,y_0)\,\aro\,\pi^{\rm EF}(X,x_0)
\]
is a   closed and  normal   subgroup scheme, and its quotient is identified with $G/I$, where $I$ is the subgroup (necessarily normal) generated by all
elements of $G$ having at least one fixed point.
\end{thm}

The \emph{proof} of 
Theorem \ref{23.06.2017--1} follows precisely the same path as that of Theorem \ref{13.09.2016--3}, except that in place of Theorem \ref{06.07.2016--1} and Proposition \ref{16.02.2018--1} (the etale case),  we apply Theorem \ref{03.11.2017--2} below and Proposition \ref{29.05.2018--1}. We shall then concentrate on proving Theorem \ref{03.11.2017--2}.

\begin{thm}\label{03.11.2017--2}Let $f:Y\longrightarrow X$ be a genuinely ramified  morphism between CPC normal varieties taking the $k$-point $y_0\in Y(k)$ to the $k$-point $x_0\in X(k)$. Then 
\[
 f_{\natural}\,:\,\pi^{\rm EF}(Y,y_0)\,\aro\,\pi^{\rm EF}(X,x_0)
\]
is a quotient morphism. 
\end{thm}

\begin{proof}We employ the criterion explained by Lemma \ref{12.12.2014--2}, so that it is enough to show that for each $\ce\in\bb{EF}(X)$, the natural morphism \[H^0(\ce)\aro H^0(f^*\ce)\]
is bijective. 
We begin by considering the special case where $\ce$ is of {\it height zero} (cf. Remark \ref{05.05.19--2}). 
This being so, let $\ce=T\ti^{\g g}E$, where ${\g g}$ is finite and etale,  $T\aro X$ is a \emph{connected} ${\g g}$-torsor, and $E$ is a representation of ${\g g}$. Then, by definition, \[H^0(X,T\ti^{\g g}E)=\{\text{${\g g}$-equivariant morphisms $\ph:T\aro E_a$}\}.\] Using proper chains in $X$, we see that for any $\ph\in H^0(X,T\ti^{\g g}E)$ and any two points $t_1,t_2$ in $T$, $\ph(t_1)$ and $\ph(t_2)$ lie on the same ${\g g}$-orbit. As $T$ is connected, this is impossible unless $\ph(t_1)=\ph(t_2)$ and we see that $E^{\g g}= H^0(X,T\ti^{\g g}E)$. Since $f^{-1}T\aro Y$ is also connected (by Proposition \ref{10.11.2017--1}), the same argument shows that  $E^{\g g}=H^0(Y,f^*(T\ti^{\g g}E))$, so that the proof of the special case is complete.

We now introduce the following set of vector-bundles on $X$: 
\[ 
\g S:=\left\{\cv \,:\,\begin{array}{c}\text{for any $\tau\in H^0(f^*\cv)$, its image $F_Y^*(\tau)$ in $H^0(F_Y^*f^*\cv)$ }\\\text{ belongs to the image of $f^*\,:\,
H^0(F_X^*\cv)\,\longrightarrow\, H^0(F_Y^*f^*\cv)$}\end{array}\right\}\, .\] 
We note that if $\cv\in\bb{EF}(X)$ is of height one---in the sense of the discussion after Definition \ref{15.06.2018--2}---then $\cv\in\g S$ because in this case $F_X^*\cv$ is of height zero and therefore  
$H^0(F_X^*\cv)\stackrel\sim\aro H^0(f^*F_X^*\cv)$.
This last observation, together with a simple induction argument, shows that the lemma is a consequence of the:

\vs
\noindent\emph{Claim.} If $\ce\in \g S$, then the pull-back $H^0(\ce)\,\longrightarrow\, H^0(f^*\ce)$ is bijective. 

Let then $\tau$ be a section of $f^*\ce$; by assumption, we have
\begin{equation}\label{dag1}
F_Y^*(\tau)\,=\,f^*(\si)\, ,\ \ \text{ with }\ \si\,\in\, H^0(F_X^*\ce)\, .
\end{equation}
If $X_0\,\subset\, X$ stands for the smooth locus of $X$ (so that $\mm{codim}
(X\setminus X_0;X)\,\ge\,2$), we contend that 
\begin{equation}\label{ddag1}
\si|_{X_0}\,=\,F^*_{X_0}(\rho)\, ,\ \ \text{ for some }\rho\,\in\, H^0(\ce|_{X_0})\, .
\end{equation}
Let $\na$ be the canonical connection on $F_{X}^*(\ce)|_{X_0}\,=\,F_{X_0}^*(\ce|_{X_0})$ \cite[Theorem 5.1, p.~370]{katz}. Denote by $X_1$ and open dense subset of $X$ such that the restriction of $f$ to $Y_1:=f^{-1}(X_1)$ is etale.  Then, 
\[
\na (\si|_{X_1})\,=\,0
\] 
since $f^*(\si|_{X_1})=F_Y^*(\tau)|_{Y_1}$ and $f^*\Om_{X_1}^1\simeq\Om_{Y_1}^1$. Hence, $\na(\si|_{X_0})=0$, and Cartier's theorem \cite[Theorem 5.1, p.370]{katz} guarantees the existence of $\rho$ as in \eqref{ddag1}. 
Since $X$ is normal, $\rho$ extends to $\ov\rho\,\in\, H^0(\ce)$ 
\cite[III, Corollary 3.5]{SGA2}, and   \eqref{ddag1}   gives
\begin{equation}\label{ep1}
\si|_{X_0}\,=\,F_{X_0}^*(\ov\rho|_{X_0})\, .
\end{equation}
From this point, letting $Y_0=f^{-1}(X_0)$, equations \eqref{dag1} and
\eqref{ep1} show that
\begin{equation}\label{15.06.2018--1}F_Y^*(\tau)|_{Y_0}\,=\,
F_Y^*(f^*(\ov\rho))|_{Y_0}\, .
\end{equation}
 As $Y_0\subset Y$ is dense in the variety $Y$, equation \eqref{15.06.2018--1} guarantees that 
\[
F_Y^*(\tau)\,=\,F_Y^*(f^*(\ov\rho))\,.
\]
As $F_Y^*\,:\,H^0(\cv)\,\longrightarrow\, H^0(F_Y^*\cv)$
is always injective for a vector-bundle $\cv$, we conclude that $f^*(\ov\rho)
\,=\,\tau$. This finishes the proof of the Claim.
\end{proof}

\begin{exs} 
\begin{enumerate}[(1)]
\item Let $X=\PP(q_0,\ldots,q_r)$ be the weighted projective space associated to the set of prime-to-$p$ positive integers $q_0,\ldots,q_r$ \cite[1.2.2]{dolgachev81}. By Theorem \ref{23.06.2017--1},  the group scheme $\pi^{\rm EF}(X)$ vanishes because $\pi^{\rm EF}(\PP^r)=0$ \cite[Corollary, p.93]{nori82}.
\item The morphism in Theorem \ref{03.11.2017--2} can fail to be an isomorphism  as the following argument shows.  
We suppose $p\not=2$. Let $A$ be an abelian surface and write $\iota:A\to A$ for the inversion morphism; it induces an action of $G=\ZZ/2\ZZ$ on $A$. If $\si:\tilde A\to A$ stands for the blowup of the 16 fixed points of $\iota$, the action  extends to an action  on $\tilde A$;  and the quotient $K:=G\backslash \tilde A$ is called the Kummer surface of $A$  (cf. \cite[10.5]{badescu01} for details). Now  $\pi^{\rm EF}(\tilde A)\simeq\pi^{\rm EF}(A)$ by \cite[Proposition 8, p.92]{nori82} while $K$, being a K3 surface,    has a vanishing essentially finite fundamental group scheme  (see Lemma 2.3 and Remark 2.4 in \cite{biswas-dos_santos13}). 
\end{enumerate}
\end{exs}

\end{document}